\documentclass[reqno, 12pt]{article}

\pdfoutput=1

\usepackage{enumerate}
\usepackage{latexsym}
\usepackage[centertags]{amsmath}
\usepackage{amsfonts}
\usepackage{amssymb}
\usepackage{amsthm}
\usepackage{newlfont}
\usepackage{graphics}
\usepackage{color}
\usepackage{float}
\usepackage{diagbox}
\usepackage{xcolor}
\usepackage{tocloft}
\usepackage{titlesec}
\usepackage{subcaption}

\usepackage{tikz}          
\usetikzlibrary{tikzmark}

\usepackage{booktabs}
\usepackage{algorithm}
\usepackage{algorithmicx}
\usepackage{algpseudocode}
\usepackage{url}
\usepackage{hyperref}
\usepackage{longtable}
\usepackage{rotating}
\usepackage{multirow}
\usepackage{extarrows}
\usepackage[sort,compress,numbers]{natbib}
\usepackage[utf8]{inputenc}

\newtheorem{thm}{Theorem}[section]
\newtheorem{cor}[thm]{Corollary}
\newtheorem{lem}[thm]{Lemma}
\newtheorem{prop}[thm]{Proposition}

\theoremstyle{mydefinition}
\newtheorem{dfn}[thm]{Definition}
\theoremstyle{myremark}
\newtheorem{rem}[thm]{Remark}
\newtheorem{exa}[thm]{Example}

\newtheorem{prob}[thm]{Problem}

\allowdisplaybreaks[4]

\makeatletter
\let\c@algorithm\c@thm
\makeatother
\numberwithin{algorithm}{section}

\def\Z{\mathbb{Z}}

\newcommand{\h}{{\mathrm{H}}}
\newcommand{\ph}{{\mathrm{PH}}}

\renewcommand{\P}{{\mathbb{P}}}

\title{Unimodular Equivalence of Integral Simplices}

\author{Feihu Liu$^{\color{blue} \dag}$, Sihao Tao$^{\color{blue} \S}$, and Guoce Xin$^{\color{blue} \P}$
\\[2mm]
{\small $^{\color{blue} \dag, \S, \P}$ School of Mathematical Sciences,}\\[-0.8ex]
{\small Capital Normal University, Beijing, 100048, P.R.~China}\\
{\small {\color{blue} $^\dag$} Email address: liufeihu7476@163.com}\\
{\small {\color{blue} $^\S$} Email address: sihao\_tao@cnu.edu.cn}\\
{\small {\color{blue} $^\P$} Email address: guoce\_xin@163.com}
}

\date{\today}

\begin{document}

\maketitle

\begin{abstract}
Testing the unimodular equivalence of two full-dimensional integral simplices can be reduced to testing unimodular permutation (UP) equivalence of two nonsingular matrices. We conduct a systematic study of UP-equivalence, which leads to the first average-case quasi-polynomial time algorithm, called \texttt{HEM}, for deciding the unimodular equivalence of $d$-dimensional integral simplices, as well as achieving a polynomial-time complexity with a failure probability less than $2.5 \times 10^{-7}$.
A key ingredient is the introduction of the \emph{permuted Hermite normal form} and its associated \emph{pattern group}, which streamlines the UP-equivalence test by comparing canonical forms derived from induced coset representatives. We also present an acceleration strategy based on Smith normal forms.
As a theoretical by-product, we prove that two full-dimensional integral simplices are unimodularly equivalent if and only if their $n$-dimensional pyramids are unimodularly equivalent. This resolves an open question posed by Abney-McPeek et al.
\end{abstract}

\noindent
\begin{small}
\end{small}

\noindent
\begin{small}
\end{small}

\tableofcontents

\section{Introduction}\label{sec:intro}
This paper presents the first practical algorithm for testing the unimodular equivalence of
full-dimensional integral simplices that run in \emph{quasi-polynomial time on
average} and in \emph{provable polynomial time} with failure probability less than
$2.5\times 10^{-7}$. 

We first introduce some basic concepts, then present our two major contributions.

\subsection{Basic concepts}
Let $\mathcal{P}$ be a full-dimensional polytope in $\mathbb{R}^{d}$. Ehrhart theory~\cite{Ehrhart62} studies the function counting the number of integer lattice points in the $n$-th dilation of $\mathcal{P}$, defined as
\[\mathrm{ehr}(\mathcal{P},n)=|n\mathcal{P}\cap \mathbb{Z}^d|,\ \ \ \ \ n=1,2,\ldots,\]
where $n\mathcal{P}:=\{n \alpha : \alpha \in \mathcal{P}\}$. A polytope $\mathcal{P}$ is called an \emph{integral polytope} if its vertices have integer coordinates, and a \emph{rational polytope} if its vertices lie in $\mathbb{Q}^{d}$. For a rational polytope $\mathcal{P}$, the \emph{denominator} of $\mathcal P$ is the smallest positive integer $k$ such that $k\mathcal{P}$ is an integral polytope.

Ehrhart theory has attracted considerable attention, yielding many results concerning the function $\mathrm{ehr}(\mathcal{P},n)$; we refer the reader to~\cite{BeckRobins} and~\cite[Chapter 4]{Stanley-Vol-1}. A fundamental result~\cite{Ehrhart62} states that for any rational polytope $\mathcal{P}$ with denominator $k$, the function $\mathrm{ehr}(\mathcal{P},n)$ is a \emph{quasi-polynomial} in $n$ of period $k$. This means
\[\mathrm{ehr}(\mathcal{P},n)=c_d(n)n^d+\cdots +c_1(n)n+c_0(n),\]
where each $c_i(n)$ is a periodic function with period dividing $k$, and $c_d(n)$ is not identically zero. This quasi-polynomial is the \emph{Ehrhart quasi-polynomial} of $\mathcal{P}$. When $\mathcal{P}$ is integral, $\mathrm{ehr}(\mathcal{P},n)$ is a polynomial of degree $d$, called the \emph{Ehrhart polynomial} of $\mathcal{P}$.

We now recall key definitions, following~\cite{Turner-Wu} (see also~\cite{Haase-McAllister,Greenberg}).

\begin{dfn}
Two rational polytopes $\mathcal{P}, \mathcal{Q}\subseteq \mathbb{R}^d$ are \emph{Ehrhart-equivalent} if $\mathrm{ehr}(\mathcal{P},n)=\mathrm{ehr}(\mathcal{Q},n)$ for all positive integers $n$; equivalently, $\mathcal{P}$ and $\mathcal{Q}$ share the same Ehrhart quasi-polynomial.
\end{dfn}

\begin{dfn}
An \emph{affine unimodular transformation} $U: \mathbb{R}^d\rightarrow \mathbb{R}^d$ is a map of the form
\[U(\mathbf{x})=A\mathbf{x}+\mathbf{b},\ \ \mathbf{x}\in \mathbb{R}^d,\]
where $A \in \mathrm{GL}_d(\mathbb{Z})$ (the group of integer $d\times d$ matrices with determinant $\pm 1$) and $\mathbf{b}\in \mathbb{Z}^d$. The set of such transformations is denoted $\mathrm{GL}_d(\mathbb{Z})\ltimes \mathbb{Z}^d$.
\end{dfn}

\begin{dfn}
Two rational polytopes $\mathcal{P}, \mathcal{Q}\subseteq \mathbb{R}^d$ are \emph{unimodularly equivalent} if there exists an affine unimodular transformation $U\in \mathrm{GL}_d(\mathbb{Z})\ltimes \mathbb{Z}^d$ satisfying $U(\mathcal{P})=\mathcal{Q}$.
\end{dfn}

A set $S\in \mathbb{R}^d$ is called \emph{relatively open} if it is open inside its affine span \cite{Erbe-Haase-Santos}.

\begin{dfn}
Two rational polytopes $\mathcal{P}, \mathcal{Q} \subseteq \mathbb{R}^d$ are \emph{$\mathrm{GL}_d(\mathbb{Z})$-equidecomposable} if there exist relatively open simplices $\mathcal{P}_1,\ldots,\mathcal{P}_r$ and unimodular transformations $U_1,\ldots,U_r \in \mathrm{GL}_d(\mathbb{Z}) \ltimes \mathbb{Z}^d$ such that
\begin{align*}
\mathcal{P} = \bigsqcup_{i=1}^r \mathcal{P}_i \quad \text{and} \quad \mathcal{Q} = \bigsqcup_{i=1}^r U_i(\mathcal{P}_i),
\end{align*}
where $\bigsqcup$ denotes disjoint union.
\end{dfn}

Unimodular equivalence implies both Ehrhart equivalence and $\mathrm{GL}_d(\mathbb{Z})$-equidecomposability, and it plays a central role in Ehrhart theory. Motivated by the work of Kaibel and Schwartz~\cite{Kaibel} and Abney-McPeek et al.~\cite{Abney-Mcpeek}, we investigate algorithms for determining unimodular equivalence between full-dimensional integral simplices. The significance of developing such algorithms is highlighted in~\cite[Problem~2]{Gribanov} and~\cite{Kaibel}. Similarly, the algorithmic problem of testing affine equivalence for polytopes is considered a fundamental problem~\cite[Section~4: Problems~19--23]{Kaibel-Pfetsch}.

\begin{prob}{\em (\cite[Problem~2]{Gribanov}: Unimodular Equivalence Checking)}
Given simplices $\mathcal{P}$ and $\mathcal{Q}$, determine whether $\mathcal{P}$ and $\mathcal{Q}$ are unimodularly equivalent.
\end{prob}

To the best of our knowledge, three algorithms are currently known for unimodular equivalence checking:
\begin{enumerate}
    \item Kaibel and Schwartz \cite[Proposition 4]{Kaibel} provide an algorithm for full-dimensional simplicial polytopes with complexity $\mathcal O(d!\cdot d^2\cdot m^2)$, where $m$ denotes the number of facets (which reduces to $d+1$ for simplices).
    \item Gribanov \cite[Theorem 2]{Gribanov} presents an algorithm for two integral $\Delta$-modular (not necessarily full-dimensional) simplices (defined in Section~3.1) with complexity $d^{\log_2 \Delta} \cdot \Delta ! \cdot \mathrm{poly}(d,\Delta)$.
    \item Abney-Mcpeek et. al. \cite{Abney-Mcpeek} developed \texttt{VPEC} (Algorithm~\ref{alg:VPEC}), which achieves complexity $\mathcal O((d+1)!\cdot (d+1)^3)$.
\end{enumerate}

For a full-dimensional integral \( d \)-simplex \( \mathcal P = \text{conv}(\{ \mathbf{v}_0, \mathbf{v}_1, \dots, \mathbf{v}_d \}) \) with vertices \( \mathbf{v}_i \in \mathbb{Z}^d \), the \emph{lattice volume} (also called \emph{normalized volume}) is defined as
\[
\text{vol}_{\mathbb{Z}}(\mathcal P) = \left| \det\left( \mathbf{v}_1 - \mathbf{v}_0 | \mathbf{v}_2 - \mathbf{v}_0 | \dots | \mathbf{v}_d - \mathbf{v}_0\right) \right|.
\]
Notably, the lattice volume is invariant under unimodular transformations, which is a basic property for characterizing the unimodular equivalence of integral simplices.

\subsection{Two major contributions}

\textbf{Contribution 1:} The \texttt{HEM} algorithm for unimodular permutation (UP) equivalence.
Existing algorithms for testing unimodular equivalence of integral simplices
\cite{Kaibel,Gribanov,Abney-Mcpeek} exhibit a worst-case complexity factor of $d!$.
Our \texttt{HEM} (Hermite Normal Form-Based Equivalence Method) algorithm overcomes this barrier
by leveraging the Hermite normal form of integer matrices.

\begin{thm}\label{thm-core-contri}
Let $\mathcal{P} = \mathrm{conv}(A)$ and $\mathcal{Q} = \mathrm{conv}(B)$ be full-dimensional integral simplices in $\mathbb{R}^d$ with identical lattice volume $v$, where $A, B \in \mathbb{Z}^{d \times (d+1)}$ are integer matrices whose columns represent the vertices. Then \texttt{HEM} decides the unimodular equivalence of $\mathcal{P}$ and $\mathcal{Q}$
\begin{itemize}
    \item deterministically in time $\mathcal{O}((d+1)!\,(d+1)^{3})$,
    \item probabilistically in time $\mathcal{O}((d+1)^{6}\log^{2}v)$ with failure probability less than $2.5\times 10^{-7}$,
    \item and in \emph{average} time $\mathcal{O}((d+1)^{\log (d+1)}\log^{2}v)$.
\end{itemize}
\end{thm}

\noindent
The \texttt{HEM} algorithm introduces two key innovations: the \emph{permuted Hermite normal form} and associated \emph{pattern groups} within the symmetric group $\mathfrak{S}_n$. The core idea simplifies the UP-equivalence test by comparing canonical representatives derived from the coset decomposition induced by these pattern groups.

\medskip

\textbf{Contribution 2:} A dimension-free resolution to Abney-McPeek et al.'s open problem.

For $2$-dimensional integral simplices $\mathcal{P}$ and $\mathcal{Q}$, Abney-McPeek et al. established that $\mathcal{P}$ and $\mathcal{Q}$ are unimodularly equivalent if and only if their pyramids are unimodularly equivalent. They subsequently posed the following open problem \cite[Section~5]{Abney-Mcpeek}: Does this equivalence extend to $3$-dimensional and arbitrary $d$-dimensional simplices?
We provide an affirmative answer to this question in Theorem~\ref{Solve-OP-Simplices}; the proof is technical.

The paper is organized as follows. Section~2  we introduce some equivalence relations on matrices  and establish that two full-dimensional integral simplices are unimodularly equivalent if and only if their $n$-dimensional pyramids are unimodularly equivalent. This result provides an affirmative answer to the question posed by Abney-Mcpeek et al. Subsection~3.1 analyzes the computational complexity of the three known algorithms for detecting unimodular equivalence between simplices. Subsection~3.2 introduces our proposed algorithm \texttt{HEM}.
Section~4 presents a practical evaluation of \texttt{HEM} and provides an acceleration strategy for the algorithm based on the Smith normal form.

\section{Equivalence relations on matrices and simplices}
In studying unimodular equivalence of simplices, we find it natural to consider equivalence relations of matrices.
We first introduce classical results on Hermite normal form and Smith normal form. The former solves the (left) unimodular equivalence
and the latter solves the bi-unimodular equivalence. In between is unimodular permutation equivalence that arises from
unimodular equivalence for simplices.

\subsection{Equivalence relations on integer matrices}
We begin by recording the basic notion that underpins every equivalence discussed below.
\begin{dfn}
For a matrix $A \in \mathbb{Z}^{d \times d}$, the following are equivalent:
\begin{enumerate}
\item $\det(A) = \pm 1$.
\item The matrix $A$ can be obtained by applying a composition of elementary row (or column) operations to the identity matrix, where the allowed operations are:
    \begin{itemize}
        \item Adding an integer multiple of one row (column) to another row (column).
        \item Swapping two rows (columns).
        \item Multiplying a row (column) by $-1$.
    \end{itemize}
\end{enumerate}
Such a matrix is called \emph{unimodular}. The set of all such unimodular matrices is denoted $\mathrm{GL}_d(\Z)$.
\end{dfn}

\medskip
\noindent\textbf{Unimodular equivalence.}\
Building directly on the notion of unimodularity, we now introduce the first equivalence relation.

Two matrices $A,B\in\Z^{m\times n}$ are said to be \emph{(left) unimodularly equivalent}, written
\[
A \simeq_{\mathrm{U}} B \quad\iff\quad \exists\, U \in \mathrm{GL}_m(\Z) \text{ such that } UA = B.
\]
It is well known that $\simeq_{\mathrm{U}}$ is an equivalence relation.
Among the matrices equivalent to $A$, there is a distinguished representative, the \emph{Hermite normal form } (HNF), whose existence and uniqueness we recall next.

\begin{dfn}
Let $A\in \Z^{d \times d}$ be nonsingular.  Then there exists a unique $U \in\mathrm{GL}_d(\Z)$ such that $H:=UA$ is upper triangular with positive diagonal entries, while every off-diagonal entry of $H$ is non-negative and strictly smaller than the diagonal entry in its column.  The matrix $H$ is called the \emph{Hermite normal form} of $A$ and is denoted $\h(A)$.
\end{dfn}

Consequently, we obtain the convenient criterion:
\[
A \simeq_{\mathrm{U}} B \quad\iff\quad \h(A)=\h(B).
\]

In the following,  \( A_{i,j} \) denotes the \((i,j)\)-entry of matrix \( A \). Two immediate consequences of the definition are worth isolating.

\begin{cor}\label{cor-first-col-gcd}
Let $A\in\Z^{n\times n}$ be nonsingular.  Then we have
$$
\h(A)_{1,1} = \mathrm{GCD}\bigl( \{ A(i,1) \mid 1 \le i \le n \} \bigr),
$$
where $\mathrm{GCD}$ denotes the greatest common divisor of a set of integers.
\end{cor}

\begin{cor}\label{cor-block-HNF}
Let $M = \begin{pmatrix} A & C \\ \mathbf{0} & B \end{pmatrix}$ be nonsingular with $A$ an $r\times r$ block and $B$ an $s\times s$ block.  If $\h(M) = \begin{pmatrix} H_1 & * \\ \mathbf{0} & H_2 \end{pmatrix}$, then
$
H_1 = \h(A) \quad\text{and}\quad H_2 = \h(B).
$
\end{cor}

\def\diag{\mathop{\mathrm{diag}}}

\medskip
\noindent\textbf{Bi-unimodular equivalence}

A finer relation, called \emph{bi-unimodular equivalence}, is obtained by allowing both left and right multiplication by unimodular matrices; the resulting canonical form is the Smith normal form.
\begin{dfn}
The \emph{Smith normal form (SNF)} of a matrix $A \in \mathbb{Z}^{m \times n}$ with rank $r$ is
\[
\mathrm{SNF}_A := P A Q = \begin{pmatrix}
          \diag(d_1,d_2,\dots,d_r) &  \mathbf 0  \\
          \mathbf 0 &  \mathbf 0
        \end{pmatrix}
\]
where $P \in \mathrm{GL}_m(\mathbb{Z})$, $Q \in \mathrm{GL}_n(\mathbb{Z})$ are not unique, and the positive integers $d_1, d_2, \ldots, d_r$ satisfy the divisibility condition $d_i \mid d_{i+1}$ for $1 \leq i \leq r-1$.
\end{dfn}

A matrix $A$ is bi-unimodularly equivalent to $B$ if and only if $\mathrm{SNF}_A=\mathrm{SNF}_B$.

\medskip
\noindent\textbf{Unimodular permutation equivalence.}\

Rather than allowing full right multiplication, we next restrict the right factor to a subgroup of permutation matrices, thereby interpolating between unimodular and bi-unimodular equivalence.

Define a right action of $\mathfrak{S}_n$ on the set of $m\times n$ matrices by column permutation: for $\sigma\in\mathfrak{S}_n$ and $A=(\alpha_1\mid\cdots\mid\alpha_n)$ set
\[
A \cdot \sigma = (\alpha_{\sigma(1)}\mid\cdots\mid\alpha_{\sigma(n)}) = AP_\sigma,
\]
where $P_\sigma = (\mathbf{e}_{\sigma(1)},\dots,\mathbf{e}_{\sigma(n)})$ is the permutation matrix associated with $\sigma$.

For a subset $\mathcal I \subseteq \mathfrak{S}_n$ we define
\[
A \simeq_{\mathcal{I}} B \quad\iff\quad \exists\, U \in \mathrm{GL}_m(\Z),\; g \in \mathcal{I} \text{ such that } UA = B \cdot g \ (=BP_g).
\]
When $\mathcal I=\mathfrak{S}_n$ we write $\simeq_{\mathrm{UP}}$ and speak of \emph{unimodular permutation (UP) equivalence}.  This relation arises naturally when comparing lattice simplices up to unimodular transformation.

\begin{lem}
For any subgroup $\mathcal G \leq \mathfrak{S}_n$, the relation $\simeq_{\mathcal{G}}$ is an equivalence relation.
\end{lem}
\begin{proof}
Reflexivity, symmetry, and transitivity follow because $\mathrm{GL}_m(\Z)$ and $\mathcal G$ are groups and the two actions commute.
\end{proof}

\begin{lem}\label{lem-permute-A-B}
Suppose $A'=A \cdot\sigma_1$ and $B'=B\cdot\sigma_2$.
 If $UA =B \cdot\tau$, then $ UA' = B' \cdot (\sigma_2^{-1} \tau \sigma_1) $.
\end{lem}

\begin{proof}
The lemma follows from the following equivalences:
    \[
        UA = B \cdot \tau
        \iff UA' \cdot \sigma_1^{-1} = (B' \cdot \sigma_2^{-1}) \cdot \tau
        \iff UA' = B' \cdot (\sigma_2^{-1} \tau \sigma_1).
    \]
\end{proof}

Unlike the unimodular or bi-unimodular settings,  $\simeq_{\mathcal{G}}$ does not admit a useful normal form; instead, Section~\ref{sec-Permuted-Hermite-Normal-Form} will describe a quasi-standard form that suffices for our purposes.
We conclude this subsection with the following result, which plays a crucial role in the proof of Theorem~\ref{Solve-OP-Simplices}.

\begin{prop}\label{prop-up-pry}
Let
\[
C_1 = \begin{pmatrix}
A_1 & A_2 \\
\mathbf{0} & I_{n-d}
\end{pmatrix}
\quad \text{and} \quad
D_1 = \begin{pmatrix}
B_1 & B_2 \\
\mathbf{0} & I_{n-d}
\end{pmatrix},
\]
where $A_1$ and $B_1$ are $d \times d$ matrices, $A_2$ and $B_2$ are arbitrary matrices of appropriate sizes, and $I_{n-d}$ denotes the $(n-d) \times (n-d)$ identity matrix. Then
\begin{equation}\label{equ-up-pyr}
A_1 \simeq_{\mathrm{UP}} B_1 \quad \Longleftrightarrow \quad C_1 \simeq_{\mathrm{UP}} D_1.
\end{equation}
\end{prop}

\begin{proof}
It is clear that
\[
C_1' := \begin{pmatrix}
A_1 & \mathbf{0} \\
\mathbf{0} & I_{n-d}
\end{pmatrix} \simeq_{\mathrm{U}} C_1
\quad \text{and} \quad
D_1' := \begin{pmatrix}
B_1 & \mathbf{0} \\
\mathbf{0} & I_{n-d}
\end{pmatrix} \simeq_{\mathrm{U}} D_1.
\]
Hence, in proving~\eqref{equ-up-pyr}, we may replace $C_1$ and $D_1$ by $C_1'$ and $D_1'$, respectively.

The necessity is straightforward. Suppose there exist $U \in \mathrm{GL}_{d}(\mathbb{Z})$ and $\sigma \in \mathfrak{S}_d$ such that $UA_1 = B_1P_\sigma$. Then the identity
\[
\begin{pmatrix}
U & \mathbf{0} \\
\mathbf{0} & I_{n-d}
\end{pmatrix}
C_1' =
D_1'
\begin{pmatrix}
P_\sigma & \mathbf{0} \\
\mathbf{0} & I_{n-d}
\end{pmatrix}
\]
immediately yields $C_1^{\prime} \simeq_{\mathrm{UP}} D_1^{\prime}$.

We now prove the sufficiency. Suppose there exist $V \in \mathrm{GL}_{n}(\mathbb{Z})$ and $\tau \in \mathfrak{S}_n$ such that $VC_1^{\prime} = D_1^{\prime}P_\tau$. Write the matrices in column form as
\begin{align*}
C_1' &= \big( c_1 \mid \cdots \mid c_{d} \mid \widehat{\mathbf{e}}_1 \mid \cdots \mid \widehat{\mathbf{e}}_{n-d} \big), \\
D_1' &= \big( f_1 \mid \cdots \mid f_{d} \mid \widehat{\mathbf{e}}_1 \mid \cdots \mid \widehat{\mathbf{e}}_{n-d} \big),
\end{align*}
where, for $1 \le i \le d$, the columns $c_i$ and $f_i$ are nonzero only in their first $d$ entries.

By Lemma~\ref{lem-permute-A-B}, we may assume (after a suitable permutation of columns) that
\begin{align}
V c_i &=
\begin{cases}
f_i, & \text{for } 1 \le i \le k, \\
\widetilde{\mathbf{e}}_{i}, & \text{for } k+1 \le i \le d,
\end{cases}
\label{eq:combined1} \\
V \widehat{\mathbf{e}}_i &= f_{i+k}, \quad \text{for all } 1 \le i \le d-k, \label{eq:combined2}
\end{align}
where $k$ is an integer with $0 \le k \le d$.

From the above, we deduce the following unimodular equivalences:
\begin{align}
\big( c_1 \mid \cdots \mid c_k \mid c_{k+1} \mid \cdots \mid c_{d} \big) &\simeq_{\mathrm{U}} \big( f_1 \mid \cdots \mid f_k \mid \widetilde{\mathbf{e}}_{1} \mid \cdots \mid \widetilde{\mathbf{e}}_{d-k} \big), \label{e-c2f} \\
\big( c_1 \mid \cdots \mid c_k \mid \widehat{\mathbf{e}}_{1} \mid \cdots \mid \widehat{\mathbf{e}}_{d-k} \big) &\simeq_{\mathrm{U}} \big( f_1 \mid \cdots \mid f_k \mid \cdots \mid f_{d} \big). \label{e-f2c}
\end{align}

When $k = d$, we may assume without loss of generality that
\[
K \begin{pmatrix} A_1 \\ \mathbf{0} \end{pmatrix} = \begin{pmatrix} B_1 \\ \mathbf{0} \end{pmatrix} P_\sigma,
\]
where
\[
K = \begin{pmatrix} K_{11} & K_{12} \\ K_{21} & K_{22} \end{pmatrix} \in \mathrm{GL}_n(\mathbb{Z}), \quad \sigma \in \mathfrak{S}_d,
\]
with $K_{11} \in \mathbb{Z}^{d \times d}$ and $K_{22} \in \mathbb{Z}^{(n-d) \times (n-d)}$.
It follows that
\[
\begin{pmatrix} K_{11}A_1 \\ K_{21}A_1 \end{pmatrix} = \begin{pmatrix} B_1P_\sigma \\ \mathbf{0} \end{pmatrix}.
\]
Since $A_1$ is nonsingular, we deduce that $K_{21} = \mathbf{0}$. Therefore,
\[
\det(K) = \det(K_{11})\det(K_{22}) = \pm 1,
\]
which implies $K_{11} \in \mathrm{GL}_d(\mathbb{Z})$. From the identity $K_{11}A_1 = B_1P_\sigma$, we conclude that $A_1 \simeq_{\mathrm{UP}} B_1$.

Now suppose $k < d$. From \eqref{e-c2f}, we first derive the equivalence
\begin{align}
\label{e-c2cp}
(c_1 \mid \cdots \mid c_k \mid c_{k+1} \mid \cdots \mid c_d) \simeq_{\mathrm{U}} (c_1 \mid \cdots \mid c_k \mid \widehat{\mathbf{e}}_1 \mid \cdots \mid \widehat{\mathbf{e}}_{d-k}).
\end{align}

Note that the last $n - d$ components of each column vector $c_j$ and $f_j$ (for $j = 1, \dots, d$) are zero. Hence, for each $i = k+1, \dots, d$, the row vector $\mathbf{e}_i^\mathsf{T}$ can be expressed as an integer linear combination of the first \(d\) rows of the matrix \((c_1 \mid \cdots \mid c_k \mid \cdots \mid c_{d})\).

By inserting $\mathbf{e}_i^\mathsf{T}$ as a row into the matrix $(c_1 \mid \cdots \mid c_d)$, we obtain
\[
(c_1 \mid \cdots \mid c_k \mid c_{k+1} + \widehat{\mathbf{e}}_1 \mid \cdots \mid c_d + \widehat{\mathbf{e}}_{d-k}).
\]
Since $\widehat{\mathbf{e}}_i \neq \widehat{\mathbf{e}}_j$ for $i \neq j$, and since each $c_{k+i} + \widehat{\mathbf{e}}_i$ can be transformed into $\widehat{\mathbf{e}}_{k+1}$ via elementary row operations, we obtain
\begin{align*}
(c_1 \mid \cdots \mid c_d) &\simeq_{\mathrm{U}} (c_1 \mid \cdots \mid c_k \mid c_{k+1} + \widehat{\mathbf{e}}_1 \mid \cdots \mid c_d + \widehat{\mathbf{e}}_{d-k}) \\
&\simeq_{\mathrm{U}} (c_1 \mid \cdots \mid c_k \mid \widehat{\mathbf{e}}_1 \mid \cdots \mid \widehat{\mathbf{e}}_{d-k}),
\end{align*}
as required.

Combining the equivalences \eqref{e-f2c} and \eqref{e-c2cp}, we conclude that
\begin{align}
\label{e-cf}
\begin{pmatrix} A_1 \\ \mathbf{0} \end{pmatrix}
= (c_1 \mid \cdots \mid c_d) &\simeq_{\mathrm{UP}} (f_1 \mid \cdots \mid f_d) = \begin{pmatrix} B_1 \\ \mathbf{0} \end{pmatrix}.
\end{align}
By an argument analogous to the case $k = d$, it follows that $A_1 \simeq_{\mathrm{UP}} B_1$.
\end{proof}

\subsection{Unimodular equivalence of integral simplices}

Unimodular equivalence of integral simplices can be converted into UP-equivalence of matrices by the following lemma. It appears as \cite[Proposition 3.9]{Abney-Mcpeek} but in an equivalent form. For completeness, we provide a proof.
\begin{lem}\label{UEcone}
Let $\mathcal{P} = \mathrm{conv}(A)$ and $\mathcal{Q} = \mathrm{conv}(B)$ be  full-dimensional integral simplices in $\mathbb{R}^d$, where $A, B \in \mathbb{Z}^{d \times (d+1)}$ are integer matrices whose columns are the vertices of the corresponding simplices.
Then $\mathcal{P}$ and $\mathcal{Q}$ are unimodularly equivalent if and only if
\[
\begin{pmatrix}
    A \\
    \mathbf{1}^T
\end{pmatrix}
\simeq_{\mathrm{UP}}
\begin{pmatrix}
    B \\
    \mathbf{1}^T
\end{pmatrix},
\]
where $\mathbf{1}$ denotes the all-ones vector,  and $\cdot^T$ denotes the transpose.
\end{lem}

\begin{proof}
Recall that $\mathcal{P}$ and $\mathcal{Q}$ are unimodularly equivalent if and only if
there exist $V \in \mathrm{GL}_d(\mathbb{Z})$, $\mathbf{c} \in \mathbb{Z}^d$, and $\sigma \in \mathfrak{S}_{d+1}$ such that
the matrix
\[
K = \begin{pmatrix}
V & \mathbf{c} \\
\mathbf{0} & 1
\end{pmatrix}
\]
satisfies
\[
K \begin{pmatrix}
A \\
\mathbf{1}^T
\end{pmatrix} = \begin{pmatrix}
B \\
\mathbf{1}^T
\end{pmatrix} P_\sigma.
\]

($\Rightarrow$) [Necessity]:
Assuming unimodular equivalence, we have $\det(K) = \det(V) = \pm 1$.
Thus $K \in \mathrm{GL}_{d+1}(\mathbb{Z})$, thereby establishing the necessity.

($\Leftarrow$) [Sufficiency]:
Suppose there exist
$U = \begin{pmatrix}
V & \mathbf{c} \\
\mathbf{m}^T & u
\end{pmatrix} \in \mathrm{GL}_{d+1}(\mathbb{Z})$
with $\mathbf{c}, \mathbf{m} \in \mathbb{Z}^d$, $u \in \mathbb{Z}$,
and $\sigma \in \mathfrak{S}_{d+1}$ such that
\[
U \begin{pmatrix}
A \\
\mathbf{1}^T
\end{pmatrix} = \begin{pmatrix}
B \\
\mathbf{1}^T
\end{pmatrix} P_\sigma=\begin{pmatrix}
B P_\sigma \\
\mathbf{1}^T
\end{pmatrix} .
\]
We demonstrate that $\mathbf{m} = \mathbf{0}$ and $u = 1$.
Since $\begin{pmatrix} A \\ \mathbf{1}^T \end{pmatrix}$ is nonsingular,
the equation
\[
(\mathbf{m}^T, u) \begin{pmatrix} A \\ \mathbf{1}^T \end{pmatrix} = \mathbf{1}^T
\]
admits a unique solution. Direct substitution confirms that
$\mathbf{m} = \mathbf{0}$ and $u = 1$ satisfy this equation.
This proves the sufficiency.
\end{proof}

The following theorem solves an open question \cite[Section 5]{Abney-Mcpeek} of Abney-Mcpeek et al.
\begin{thm}\label{Solve-OP-Simplices}
Let $\mathcal{Q}_{d,1} = \mathrm{conv}(A) \subseteq \mathbb{R}^d$ and $\mathcal{Q}_{d,2} = \mathrm{conv}(B) \subseteq \mathbb{R}^d$ be arbitrary  full-dimensioal  integral simplices, where $A,B\in \mathbb{R}^{d\times (d+1)}$.

For any integer $n > d$, let $\mathcal{Q}^n_{d,1}$ and $\mathcal{Q}^n_{d,2}$ be the $n$-dimensional pyramids generated by $\mathcal{Q}_{d,1}$ and $\mathcal{Q}_{d,2}$ respectively via adjoining the elementary basis vectors $\mathbf{e}_{d+1}, \dots, \mathbf{e}_n$, such that $\mathcal{Q}^n_{d,1} = \mathrm{conv}(C)$ and $\mathcal{Q}^n_{d,2} = \mathrm{conv}(D)$ with
\begin{align*}
C &= \begin{pmatrix}
            A & \mathbf{0}_{d\times(n-d)} \\
            \mathbf{0}_{(n-d)\times(d+1)} & I_{n-d}
          \end{pmatrix}, \quad
D = \begin{pmatrix}
            B & \mathbf{0}_{d\times(n-d)} \\
            \mathbf{0}_{(n-d)\times(d+1)} & I_{n-d}
          \end{pmatrix}.
\end{align*}
Then $\mathcal{Q}_{d,1}$ and $\mathcal{Q}_{d,2}$ are unimodularly equivalent if and only if $\mathcal{Q}^n_{d,1}$ and $\mathcal{Q}^n_{d,2}$ are unimodularly equivalent.
\end{thm}

\begin{proof}
We construct the following four nonsingular matrices
\begin{align*}
A_1 = \begin{pmatrix} A \\ \mathbf{1}^T \end{pmatrix}, \quad
B_1 = \begin{pmatrix} B \\ \mathbf{1}^T \end{pmatrix}, \quad
C_1 = \begin{pmatrix} C \\ \mathbf{1}^T \end{pmatrix}, \quad \text{and} \quad
D_1 = \begin{pmatrix} D \\ \mathbf{1}^T \end{pmatrix}.
\end{align*}
By Lemma \ref{UEcone}, we need to show that $A_1 \simeq_{\mathrm{UP}} B_1$ if and only if $C_1 \simeq_{\mathrm{UP}} D_1$.
We need the following two intermediate matrices
\begin{align*}
C_1^{\prime} =\begin{pmatrix}
A_1 & \mathbf{0} \\
\mathbf{0} & I_{n-d}
\end{pmatrix}\quad\text{ and } \quad
D_1^{\prime}= \begin{pmatrix}
B_1 & \mathbf{0} \\
\mathbf{0} & I_{n-d}
\end{pmatrix}.
\end{align*}

The equivalences $C_1 \simeq_{\mathrm{U}} C_1^{\prime}$ and $D_1 \simeq_{\mathrm{U}} D_1^{\prime}$ follows by direct computation:
\begin{align*}
C_1^{\prime}
= \begin{pmatrix}
I_d & \mathbf{0}_{d \times (n-d)} & \mathbf{0}_{d \times 1} \\
\mathbf{0}_{1 \times d} & -\mathbf{1}_{1 \times (n-d)} & 1 \\
\mathbf{0}_{(n-d) \times d} & I_{n-d} & \mathbf{0}_{(n-d) \times 1}
\end{pmatrix} C_1, \quad \quad
D_1^{\prime}=  \begin{pmatrix}
I_d & \mathbf{0}_{d \times (n-d)} & \mathbf{0}_{d \times 1} \\
\mathbf{0}_{1 \times d} & -\mathbf{1}_{1 \times (n-d)} & 1 \\
\mathbf{0}_{(n-d) \times d} & I_{n-d} & \mathbf{0}_{(n-d) \times 1}
\end{pmatrix} D_1.
\end{align*}

Thus, the proof is completed by an application of Proposition \ref{prop-up-pry}.
\end{proof}

\begin{cor}{\em (\cite[Theorem 4.1]{Abney-Mcpeek})}
Let $\mathcal{Q}_{2,1}, \mathcal{Q}_{2,2} \subseteq \mathbb{R}^2$ be two arbitrary integral 2-simplices, such that
\begin{align*}
\mathcal{Q}_{2,1} = \mathrm{conv}\{(0,0), (a,b), (c,d)\}\quad \text{ and} \quad \mathcal{Q}_{2,2} = \mathrm{conv}\{(0,0), (a^{\prime},b^{\prime}), (c^{\prime},d^{\prime})\}.
\end{align*}
Now for any $n > 2$, let $\mathcal{Q}^n_{2,1}, \mathcal{Q}^n_{2,2} \subseteq \mathbb{R}^n$ be the $n$-dimensional pyramids generated by $\mathcal{Q}_{2,1}$ and $\mathcal{Q}_{2,2}$ respectively by adding the elementary basis vectors $\mathbf{e}_3, \ldots, \mathbf{e}_n$; in particular,
\begin{align*}
\mathcal{Q}^n_{2,1} &= \mathrm{conv}\{0, (a,b,\ldots,0), (c,d,\ldots,0), \mathbf{e}_3, \ldots, \mathbf{e}_n\},
\\ \mathcal{Q}^n_{2,2} &= \mathrm{conv}\{0, (a^{\prime},b^{\prime},\ldots,0), (c^{\prime},d^{\prime},\ldots,0), \mathbf{e}_3, \ldots, \mathbf{e}_n\}.
\end{align*}
Then $\mathcal{Q}_{2,1}$ and $\mathcal{Q}_{2,2}$ are unimodularly equivalent if and only if $\mathcal{Q}^n_{2,1}$ and $\mathcal{Q}^n_{2,2}$ are unimodularly equivalent.
\end{cor}

\begin{cor}
For all $n\geq d$, let $\mathcal{Q}^n_{d,1}$ and $\mathcal{Q}^n_{d,2}$ be as defined in Theorem \ref{Solve-OP-Simplices}. If $\mathcal{Q}_{d,1}$ and $\mathcal{Q}_{d,2}$ are unimodularly equivalent, then for all $n\geq d$, $\mathcal{Q}^n_{d,1}$ and $\mathcal{Q}^n_{d,2}$ are unimodularly equivalent and hence $\mathrm{GL}_n(\mathbb{Z})$-equidecomposable and Ehrhart-equivalent.
\end{cor}

\section{Unimodular equivalence checking algorithm}
We begin by reviewing three known algorithms for determining unimodular equivalence of lattice simplices.
Building on this survey, we introduce a new algorithm that decides UP-equivalence for arbitrary integer matrices.

\subsection{Existing algorithms}
To our knowledge, only three algorithms have been proposed for detecting unimodular equivalence between integral simplices.
We summarize each result below, highlighting its computational bottleneck; this analysis motivates the design of our new method.

\paragraph{Definitions and implication chain.}
We first recall standard equivalence notions for polytopes following Kaibel~\cite{Kaibel}.
Two polytopes $\mathcal{P}\subseteq\mathbb{R}^{d_1}$ and $\mathcal{Q}\subseteq\mathbb{R}^{d_2}$ are \emph{affinely equivalent} if there exists a bijective affine map $T\colon\mathrm{aff}(\mathcal{P})\to\mathrm{aff}(\mathcal{Q})$ satisfying $T(\mathcal{P})=\mathcal{Q}$.
Similarly, $\mathcal{P}$ and $\mathcal{Q}$ are \emph{projectively isomorphic} when a bijective projective transformation $T$ maps $\mathcal{P}$ to $\mathcal{Q}$, and \emph{congruent} when the isomorphism is induced by an orthogonal matrix.
Finally, two polytopes are \emph{combinatorially equivalent} if their face lattices are isomorphic.
These notions form a strict hierarchy~\cite{Kaibel}:
\[
\text{congruent}\Rightarrow\text{affinely isomorphic}\Rightarrow\text{projectively isomorphic}\Rightarrow\text{combinatorially isomorphic}.
\]

\textbf{First algorithmic complexity result.}
Kaibel and Schwartz established a complexity bound for the congruence problem of simplicial polytopes using graph isomorphism algorithms \cite{Luks-Graph}.
\begin{thm}{\em (\cite[Proposition 4]{Kaibel})}\label{K-S-Algorithm}
Given two simplicial $d$-dimensional polytopes $\mathcal{P}$ and $\mathcal{Q}$ specified by vertex coordinates ($\mathcal{V}$-descriptions),
their unimodular equivalence can be determined in $\mathcal O(d!\cdot d^2\cdot m^2)$ time,
where $m$ denotes the number of facets. In particular, for fixed dimension $d$,
the problem admits an $\mathcal O(m^2)$-time algorithm.
\end{thm}

A matrix is \emph{$\Delta$-modular} if the absolute value of every full-rank minor is at most $\Delta$,
and a polytope is \emph{$\Delta$-modular} if it admits a description $Ax\leq b$ with a $\Delta$-modular matrix $A$.
For this restricted class, Gribanov~\cite{Gribanov} recently established the following result.

\textbf{Second algorithmic complexity result:}
\begin{thm}{\em (\cite[Theorem 2]{Gribanov})}
For two $\Delta$-modular simplices $\mathcal{P}, \mathcal{Q} \subseteq \mathbb{R}^d$,
it can be determined whether they are unimodularly equivalent in time
$d^{\log_{2}\Delta} \cdot \Delta! \cdot \mathrm{poly}(d,\Delta)$.
This time bound is polynomial in $d$ for fixed~$\Delta$.
\end{thm}

When $\Delta$ is not fixed, the known algorithm~\cite{Abney-Mcpeek}
has complexity $\mathcal O((d+1)! \cdot (d+1)^3)$, which becomes prohibitive in higher dimensions. Here, we count the number of element-wise multiplications as the basic operation,
ignoring the sizes of the matrix elements. This establishes \textbf{the third algorithmic complexity result}. The procedure from~\cite{Abney-Mcpeek} (Algorithm~\ref{alg:VPEC}) operates as follows:

\algtext*{EndIf}
\algtext*{EndWhile}
\algtext*{EndFor}
\begin{algorithm}[h]
\caption{\texttt{VPEC}: Vertex Permutation Equivalence Check for Integral Simplices}\label{alg:VPEC}
\begin{algorithmic}[1]
\Require Two full-dimensional integral simplices $\mathcal{P}_1 = \mathrm{conv}(A_1)$, $\mathcal{P}_2 = \mathrm{conv}(A_2)$ in $\mathbb{R}^d$, with $A_1, A_2 \in \mathbb{Z}^{d \times (d+1)}$,
and a subset  $\mathcal{I} \subseteq \mathfrak{S}_{d+1}$ (typically $\mathcal{I} = \mathfrak{S}_{d+1}$).
\Ensure TRUE if $\mathcal{P}_1$ and $\mathcal{P}_2$ are unimodularly equivalent, FALSE otherwise.
\State Construct homogenized vertex matrices:
$H_1 \gets \begin{pmatrix} A_1 \\ \mathbf{1}^T \end{pmatrix}$, $H_2 \gets \begin{pmatrix} A_2 \\ \mathbf{1}^T \end{pmatrix}$
\If{$|\det(H_1)| \neq |\det(H_2)|$}
\State \Return FALSE.
\EndIf
\For{each permutation $\pi \in \mathcal{I}$}
    \State Compute candidate matrix: $U \gets ( H_2\cdot\pi) H_1^{-1}$
    \If {$U \in \mathbb{Z}^{(d+1)\times (d+1)}$}
        \State \Return [TRUE, $\pi$]
    \EndIf
\EndFor
\State \Return FALSE
\end{algorithmic}
\end{algorithm}

\subsection{An algorithm via Hermite normal form}
We now consider a more general setting by examining the UP-equivalence of nonsingular integer matrices with equal determinants (up to sign).

Since both the Hermite normal form (HNF) and the greatest common divisors (GCDs) of matrix columns can be computed efficiently in polynomial time, we develop in this subsection an effective procedure utilizing these tools to determine whether two nonsingular \(d \times d\) integer matrices are UP-equivalent. (The HNF of an integer matrix can be computed in \(\mathcal O(d^3\log^2 |\det A|)\) time; see the original work of \cite{Domich-Kannan-Trotter} or the comprehensive survey \cite{HNF_survey}.)

\subsubsection{Permuted Hermite normal form}\label{sec-Permuted-Hermite-Normal-Form}
The following result (for $\mathcal G = \mathfrak{S}_d$) can be used to accelerate Algorithm~\ref{alg:VPEC}.

\begin{cor}\label{cor-HNF-eq}
  Let $\mathcal G \leq \mathfrak{S}_d$ be a subgroup, and let $A,B$ be nonsingular matrices in $\Z^{d\times d}$ with HNFs $\h(A)$ and $\h(B)$, respectively. Then the following are equivalent:
  \begin{align*}
    A \simeq_{\mathcal{G}} B &\iff \h(A) \simeq_{\mathcal{G}} \h(B) \\
    &\iff \h(A) = \h(B\cdot g) \text{ for some } g\in\mathcal{G}.
  \end{align*}
\end{cor}
\begin{proof}
By definition of the HNF, we have $A \simeq_{\mathrm{U}} \h(A)$ and $B \simeq_{\mathrm{U}} \h(B)$. Since $\simeq_{\mathrm{U}}$ implies $\simeq_{\mathcal{G}}$, it follows that
$$A \simeq_{\mathcal{G}} B \iff \h(A) \simeq_{\mathcal{G}} \h(B).$$
Moreover, there exists $g\in\mathcal{G}$ such that
$$A \simeq_{\mathcal{G}} B \iff A \simeq_{\mathrm{U}} (B\cdot g) \iff \h(A) = \h(B\cdot g).$$
This completes the proof.
\end{proof}

\def\id{\mathrm{id}}
Now fix an ordering of elements in $\mathfrak S_d$ as $\id = \sigma_1, \dots, \sigma_{d!}$ such that each $\sigma_{i+1}=\sigma_i \tau_i$, where $\tau_i$ is an adjacent transposition $(j, j+1)$ with $\tau_0=\id$ denoting the identity permutation.

Corollary \ref{cor-HNF-eq} enables efficient computation of $H_i = \h(B\cdot \sigma_{i})$ by observing that $H_i \cdot (j,j+1)$ is nearly in HNF and $H_{i+1}=\h(H_i \cdot (j,j+1))$. This approach is formalized in Algorithm~\ref{alg-VVV}.

\algtext*{EndIf}
\algtext*{EndWhile}
\algtext*{EndFor}
\begin{algorithm}[h]
\caption{\texttt{EHEM} (Exhaustive Hermite Equivalence Method)}
\begin{algorithmic}[1]
\Require  $A, B \in \mathbb{Z}^{d \times d}$ (nonsingular with $|\det(A)|=|\det(B)|$)
\Ensure   $[\text{TRUE}, \tau]$ if $A\simeq_{\mathrm{U}} B\cdot\tau$, FALSE otherwise.
    \State   Compute $F=\h(A)$ and set $H_0:=B$
    \For{$i = 0$ \textbf{to} $d!-1$}
    \State  $H_{i+1} \gets \h(H_{i}\cdot \tau_i)$
    \If{$H_{i+1}=F$} \Return [TRUE, $\sigma_{i+1}$]
    \EndIf
    \EndFor
\State \Return FALSE.
\end{algorithmic}\label{alg-VVV}
\end{algorithm}

\begin{prop}
Algorithm~\ref{alg-VVV} decides in time $\mathcal O(d!\cdot d^3)$ whether two nonsingular $d\times d$ matrices $A$ and $B$ are UP-equivalent.
\end{prop}

\begin{proof}
Correctness follows from Corollary~\ref{cor-HNF-eq}. For complexity analysis, observe that the dominant cost arises from computing $H_{i+1}$.
Partition $H_i$ to obtain
$$
H_i \cdot (j,j+1)=\begin{pmatrix}
D_1 & * &* \\
\mathbf 0 &   D  & * \\
\mathbf 0 & \mathbf 0 & D_2
\end{pmatrix} \text{ with } D= \begin{pmatrix}
                                  h_{j,j+1}  & h_{j,j} \\
                                 h_{j+1,j+1}& 0
                               \end{pmatrix}.
$$
By Corollary~\ref{cor-block-HNF}, the diagonal entries of $H_i$ and $H_{i+1}$ differ only at $\h(D)$. Using the Euclidean algorithm,
we can find $U_2 \in \mathrm{GL}_2(\mathbb{Z})$ such that
\[
  \begin{pmatrix}
     I_{j-1} & \mathbf 0 & \mathbf 0 \\
     \mathbf 0 & U_2 & \mathbf 0 \\
     \mathbf 0 & \mathbf 0 &  I_{d-j-1}
   \end{pmatrix} \begin{pmatrix}
D_1 & *&* \\
\mathbf 0 &   D  & * \\
\mathbf 0 & \mathbf 0 & D_2
\end{pmatrix} = \begin{pmatrix}
D_1 & *&* \\
\mathbf 0 &   \h(D)  & * \\
\mathbf 0 & \mathbf 0 & D_2
\end{pmatrix}:= H' ,
\]  where $\h(D)=\begin{pmatrix}
     g_1 & * \\
      0 & g_2
\end{pmatrix}$ satisfies $g_1 = \operatorname{GCD}(h_{j,j+1}, h_{j+1,j+1})$ and $h_{j,j} h_{j+1,j+1} = g_1 g_2$.
This step modifies elements in rows $j$ and $j+1$ using only \(4(d-j+1)\) operations. Note that $H_{i+1}\neq F$ if the diagonal entries of $H'$ differ from those of $F$, in which case further computation can be avoided.

It remains to perform elementary row operations on $H'$ so that the column entries become positive and less than the corresponding diagonal entries.
Note that $D_1$ and $D_2$ are already in HNF, so we only need to modify $H'_{k,l}$ for $k \geq j$ and $l \leq j+1$.
It is straightforward to verify that the cost $T_k$ associated with column $k$ is given by:
$$
T_k = \begin{cases}
  (k-1) \cdot (d-k+1), & \text{if } j \leq k \leq j+1;\\
  (j+1) \cdot (d-k+1), & \text{if } j+2 \leq k \leq d.
\end{cases}
$$

Thus, the total cost aside from one GCD computation is
\begin{align*}
T(j) &= 4(d-j+1) + \sum_{k=j}^{d} T_k \\
     &= \frac{1}{2} \left(d j - j^{2} + d + 2 j - 1\right) \left(d - j + 2\right).
\end{align*}
By calculus, the cubic polynomial $T(j)$ in $j$ attains its maximum at approximately $j = \frac{d}{3}$.
Hence, $T(j) \leq \frac{2}{27}d^3 + \mathcal{O}(d^2)$.

The proposition then follows because the operations on column $l$ can be performed modulo $H'_{l,l}$.
\end{proof}

Although the aforementioned algorithm has the same complexity as Algorithm~\ref{alg:VPEC}, it can mitigate the impact of large element sizes in practical applications, especially when the determinants are large.

For UP-equivalence, we can define the standard form of $A$ as the smallest $H_i$ with respect to a total order. Here, we typically use the \emph{reverse lexicographic order} on matrices, denoted by $\le_{\text{rlex}}$. For two matrices $A$ and $B$ of the same size, we define $A \le_{\text{rlex}} B$ if either $A = B$, or when scanning the matrices row by row from top to bottom and within each row from right to left, the first entry where they differ satisfies $A_{i,j} < B_{i,j}$.

However, obtaining such a standard form is difficult. Instead, we introduce permuted Hermite normal forms, which serve a similar role to the HNF. To proceed, we first establish some notation.

For any $g \in \mathfrak{S}_n$ and an $n \times n$ matrix $A$, denote by
$$
g \circ A = P_g A P_g^{-1}.
$$
This defines an action of $\mathfrak{S}_n$ on the set of $n \times n$ matrices.

\begin{dfn}
Let $\mathbf{r} = (r_1, \dots, r_s)$ and $\mathbf{m} = (m_1, \dots, m_s)$ be two sequences of positive integers with $r_i < r_{i+1}$ for $1 \leq i < s$. Let $d = m_1 + \cdots + m_s$ and
define $\mathcal{M}(\mathbf{r}, \mathbf{m})$ to be the set of all $d \times d$ matrices of the form
\begin{equation}\label{form-PHA}
\begin{pmatrix}
    r_1 I_{m_1}    & B_{12} & \cdots & B_{1,s-1} & B_{1s} \\
    \mathbf{0} & r_2 I_{m_2} & \cdots & B_{2,s-1} & B_{2s} \\
    \vdots & \vdots & \ddots & \vdots & \vdots \\
    \mathbf{0} & \mathbf{0} & \cdots & r_{s-1} I_{m_{s-1}} & B_{s-1,s} \\
    \mathbf{0} &  \mathbf{0} &  \cdots &  \mathbf{0} & r_s I_{m_s}
\end{pmatrix},
\end{equation}
where each $B_{ij}$ ($1 \le i < j \le s$) is an arbitrary integer matrix of size $m_i \times m_j$.
\end{dfn}

For any $N \in \mathcal{M}(\mathbf{r}, \mathbf{m})$, we define the \emph{pattern group} of $N$ by
\[
\mathrm{G}(N) := \mathfrak{S}_{(m_1, \ldots, m_{s-1})} := \mathfrak{S}_{w_1} \times \cdots \times \mathfrak{S}_{w_{s-1}} \leq \mathfrak{S}_d,
\]
where $w_i = \{ m_1 + \dots + m_{i-1} + 1, \dots, m_1 + \dots + m_i \}$ and $m_0 = 0$.
Note that we have excluded the last entry $m_s$, so that $g \in \mathrm{G}(N)$ acts trivially on $w_s$.

Suppose $N \in \mathcal{M}(\mathbf{r}, \mathbf{m})$ is in HNF. Then it is clear that for any $g \in \mathrm{G}(N)$, the matrix $g \circ N = P_g N P_g^{-1}$ is also in HNF and belongs to $\mathcal{M}(\mathbf{r}, \mathbf{m})$.
Therefore, we consider the orbit of $N$ under $\mathrm{G}(N)$, denoted by $\mathrm{O}(N) = \{ g \circ N : g \in \mathrm{G}(N) \}$. Clearly, $\mathrm{O}(N) = \mathrm{O}(g \circ N)$ for any $g \in \mathrm{G}(N)$.

There exists a straightforward procedure to compute the smallest representative $N_0 = g_0 \circ N = \min \mathrm{O}(N)$ with respect to the \emph{reverse-block lexicographic order} $\leq_{\text{rblex}}$. Specifically, for $M, M' \in \mathcal{M}(\mathbf{r},\mathbf{m})$, we write $M <_{\text{rblex}} M'$ if there exists an index $k \in \{s-1,s-2,\dots,1\}$ such that:
\begin{itemize}
    \item[(i)] For all $l$ with $k < l \le s-1$, the blocks satisfy $B_{l,j} = B_{l,j}'$ for every $j = l+1, \dots, s$;
    \item[(ii)] The submatrix $C_k$, formed by the rows indexed by $w_k$ in $M$, is smaller than the corresponding submatrix $C_k'$ under the order $<_{\text{rlex}}$, i.e., $C_k <_{\text{rlex}} C_k'$.
\end{itemize}
If no such $k$ exists, then $M = M'$.

Given a matrix $N$ of the form \eqref{form-PHA}, we may compute the smallest representative $N_{0} = g \circ N$ for some $g \in \mathrm{G}(N)$ under $\leq_{\text{rblex}}$ via Algorithm~\ref{alg:RBLM_sort}.

\begin{algorithm}[h]
\caption{Reverse-Block Lexicographic Matrix Sort (RBLM Sort)}
\label{alg:RBLM_sort}
\begin{algorithmic}[1]
\Procedure{\texttt{RBLMSort}}{$N$}
    \For{$i = s-1, s-2, \dots, 1$}
       \State Sort the rows of $C_i$ under $\leq_\text{rlex}$ to obtain $P_{\sigma_i}N = (C_1 ; \cdots ; (C_i)_{\text{sorted}}; \cdots ; C_s)$.
       \State $N \gets P_{\sigma_i} N P^{-1}_{\sigma_i} = P_{\sigma_i} N P_{\sigma^{-1}_i}$.
    \EndFor
    \State \Return $[N, \sigma^{-1}_{s-1} \cdots \sigma^{-1}_{1}]$
\EndProcedure
\end{algorithmic}
\end{algorithm}

\begin{dfn}
Let $A \in \mathbb{Z}^{d \times d}$ be a nonsingular matrix. The \emph{permuted Hermite normal form (PHNF)} of $A$, denoted $\mathrm{PH}(A)$,
is defined as the set of matrices $N_0 = \h(A \cdot g)$ for some $g \in \mathfrak{S}_d$, satisfying:
\begin{itemize}
\item $N_0 \in \mathcal{M}(\mathbf{r}, \mathbf{m})$;
\item In each row, the first positive entry $r_i$ is less than or equal to every other positive entry in that row;
\item $N_0$ is the smallest element in its orbit $\mathrm{O}(N_0) = \{P_g N_0 P_g^{-1} : g \in \mathrm{G}(N_0)\}$ under the order $\leq_{\text{rblex}}$.
\end{itemize}
\end{dfn}

\begin{exa}
Let $A \in \mathbb{Z}^{6 \times 6}$ such that $$\h(A) =
 \begin{pmatrix}
1 & 0 &  \tikzmark{1_greenTL}1 &3 & 4 &\tikzmark{1greenTR} 970 \\
0 & 1 & 3 & 2 & 1 & 970\tikzmark{1_greenBR} \\
0 & 0 & 5 & 0 & 0 & \tikzmark{1_top}654 \\
0 & 0 & 0 & 5 & 0 & 290\\
0 & 0 & 0 & 0 & 5 & \tikzmark{1_bottom}5695 \\
0 & 0 & 0 & 0 & 0 & 12314
\end{pmatrix}
\begin{tikzpicture}[overlay, remember picture]
\draw[red, thick]
([xshift=21pt, yshift=10pt]pic cs:1_top) rectangle
([xshift=-1.5pt, yshift=-1.4pt]pic cs:1_bottom);
\draw[green, thick]
([xshift=-3pt, yshift=9pt]pic cs:1_greenTL) rectangle
([xshift=3pt, yshift=-1.4pt]pic cs:1_greenBR);
\end{tikzpicture}.$$
Then $\h(A) \in \mathcal{M}(\mathbf{r}, \mathbf{m})$ with $\mathbf{r} = (1,5, 12314)$ and $\mathbf{m} = (2,3,1)$.  Moreover,
\[
\mathrm{G}(\h (A)) = \mathfrak{S}_{\{1,2\}}\times\mathfrak{S}_{\{3,4,5\}},
\]
and
\[
N_0 = g\circ \h(A)  = \begin{pmatrix}
1 & 0 &  \tikzmark{greenTL}2 &3 & 1 &\tikzmark{greenTR} 970 \\
0 & 1 & 3 & 1 & 4 & 970\tikzmark{greenBR} \\
0 & 0 & 5 & 0 & 0 & \tikzmark{top}290 \\
0 & 0 & 0 & 5 & 0 & 654\\
0 & 0 & 0 & 0 & 5 & \tikzmark{bottom}5695 \\
0 & 0 & 0 & 0 & 0 & 12314
\end{pmatrix}
\begin{tikzpicture}[overlay, remember picture]
\draw[red, thick]
([xshift=20pt, yshift=10pt]pic cs:top) rectangle
([xshift=-1.5pt, yshift=-1.4pt]pic cs:bottom);
\draw[green, thick]
([xshift=-3pt, yshift=9pt]pic cs:greenTL) rectangle
([xshift=3pt, yshift=-1.4pt]pic cs:greenBR);
\end{tikzpicture} \in \ph(A),
\]
where $g = (1,2)(3,4) \in \mathrm{G}(\h(A))$.
\end{exa}

\begin{thm}
The set $\mathrm{PH}(A)$ is nonempty, and Algorithm~\ref{alg:find-ph} computes a \emph{PHNF} of $A$ in polynomial time, specifically in $\mathcal{O}(d^4 \log^2 |\det A|)$ time.
\end{thm}

\begin{proof}
We proceed by induction on $d$. The base case $d=1$ is trivial. Assume the theorem holds for $d = m-1$, and consider $d = m$.
Let column $k$ be the leftmost column with the smallest column GCD, denoted by $a$.

By Corollary~\ref{cor-first-col-gcd}, we have
\[
\h(A \cdot (k,1)) = \begin{pmatrix}
a & * \\
\mathbf{0} & D
\end{pmatrix}.
\]

Since $D \in \mathbb{Z}^{(d-1) \times (d-1)}$, the induction hypothesis together with Corollary~\ref{cor-block-HNF} imply that there exists $\sigma \in \mathfrak{S}_{1} \times \mathfrak{S}_{\{2,\dots,d\}}$ such that
\[
H = \h\bigl(A \cdot ((k,1)\sigma)\bigr) = \begin{pmatrix}
a & * \\
\mathbf{0} & D'
\end{pmatrix},
\]
where $D'$ is a PHNF. Clearly, $a \le H_{i,i}$ for all $i$, and $a$ is the smallest nonzero entry in the first row. If $a = H_{i,i}$ for some $i$, then $a = H_{i,i} > H_{1,i}$, which forces $H_{1,i} = 0$. Therefore, $H$ is a PHNF, and the conclusion follows.

Following the inductive strategy of swapping the column with the smallest GCD to the first position, we design Algorithm~\ref{alg:find-ph}. Since the cost of HNF computations dominates, the algorithm runs in polynomial time, specifically in $\mathcal{O}(d^4 \log^2 |\det A|)$ time.
\end{proof}

For convenience, let $\mathrm{CPH}(A)$ denote the particular PHNF of $A$ computed by Algorithm~\ref{alg:find-ph}.

\begin{algorithm}[H]
\caption{Computing a Permuted Hermite Normal Form}
\label{alg:find-ph}
\begin{algorithmic}[1]
\Procedure{\texttt{CPH}}{$N$, $i$, $g$}
    \If{$i = d$}
        \State $[N,h]\gets \texttt{RBLMSort}(N)$
        \State \Return $[N,gh]$
    \EndIf
    \State $k \gets \arg\min\limits_{j=i,\dots,d} \mathrm{GCD}\{N_{i,j}, \dots, N_{d,j}\}$ \Comment{Find the leftmost column with minimum GCD}
    \State Compute Hermite normal form $N \gets \h(N\cdot (i,k))$
    \State $g \gets g(i,k)$
    \State \Return $\texttt{CPH}(N, i+1, g)$ \Comment{Recursive call on the next column}
\EndProcedure

\Statex
\State \textbf{Initial call:} $\text{\texttt{CPH}}(\h(A), 1, \id)$
\end{algorithmic}
\end{algorithm}

\begin{thm}\label{thm-O-deter}
Let $A,B$ be nonsingular matrices in $\Z^{d\times d}$ and let $N_0\in\ph(B)$. Suppose we have a right coset decomposition $\mathfrak{S}_d = \bigsqcup_{i=1}^{\ell} \mathrm G(N_0)g_i$ with respect to $\mathrm G(N_0)$. Then
\begin{align*}
  A \simeq_{\mathrm{UP}} B  & \iff \h(A\cdot g_i^{-1})\in \mathrm O(N_0) \text{ for some } i \in \{1,\ldots,\ell\} \\
  & \iff \mathrm{G}( \h(A\cdot g_i^{-1})) =\mathrm{G}(N_0)  \text { and }  \min\mathrm{O}( \h(A\cdot g_i^{-1}))=N_0   \text{ for some } i \in \{1,\ldots,\ell\}.
\end{align*}
\end{thm}

\begin{proof}
The second equivalence being clear, we now prove the first.

($\Rightarrow$) [Necessity]: Since $B \simeq_{\mathrm{UP}} N_0$, we have $UA=N_0\cdot g$ for some $U\in \mathrm{GL}_n(\mathbb{Z})$ and some $g \in \mathfrak{S}_d$. Write $g = g_0 g_i$ with $g_0 \in \mathrm G(N_0)$. Then
\[
U(A\cdot g_i^{-1})=N_0\cdot g_0.
\]
This implies that $\h(A\cdot g_i^{-1})= \h(N_0\cdot g_0)= g_0^{-1} \circ N_0$. Hence $\h(A\cdot g_i^{-1})\in \mathrm O(N_0)$.

($\Leftarrow$) [Sufficiency]: If $\h(A\cdot g_i^{-1})\in \mathrm O(N_0)$, then $ \h(A\cdot g_i^{-1})  \simeq_{{\mathrm {UP}}} N_0 $. Since $ \h(A\cdot g_i^{-1})  \simeq_{{\mathrm {UP}}} A $ and $N_0 \simeq_{{\mathrm {UP}}} B $, the result follows.
\end{proof}

\subsubsection{The \texttt{HEM} Algorithm}

Based on the preceding analysis, we design a new algorithm called \texttt{HEM} (HNF-Based Equivalence Method).

\begin{thm}
Let $A,B$ be nonsingular matrices in $\Z^{d\times d}$ with $|\det(A)|=|\det(B)|$.
The UP-equivalence of $A$ and $B$ can be checked by the \emph{\texttt{HEM}} Algorithm~\ref{alg-HEM} with time complexity $\mathcal O(d!\cdot d^3 )$.
\end{thm}
\begin{proof}
The \texttt{HEM} algorithm begins by computing $\mathrm{CPH}(A)$ and $\mathrm{CPH}(B)$ using the polynomial-time Algorithm~\ref{alg:find-ph}. Steps 2--5 give a coset decomposition of $\mathfrak S_n$. Within each coset, the minimum representative can be computed by sorting the rows using Algorithm \texttt{RBLMSort}.  The correctness of the remaining steps is guaranteed by Theorem~\ref{thm-O-deter}.

The worst-case scenario occurs when $\mathrm{G}(N)=\mathrm{G}(N^\prime)=\{\id\}$, forcing the algorithm to examine all $d!$ permutations. Using \texttt{EHEM} yields an overall complexity of $\mathcal O(d!\cdot d^3)$.
\end{proof}

    \algtext*{EndIf}
\algtext*{EndWhile}
\algtext*{EndFor}
\begin{algorithm}[h]
\caption{\texttt{HEM} (HNF-Based Equivalence Method)}
\label{alg-HEM}
\begin{algorithmic}[1]
\Require  $A, B \in \mathbb{Z}^{d \times d}$ (nonsingular with $|\det(A)|=|\det(B)|$)
\Ensure   $[\text{TRUE}, (A, \sigma), (B,\tau)]$ if $A \cdot \sigma \simeq_{\mathrm{U}} B \cdot\tau$, FALSE otherwise.
    \State  $[N,g] \gets \texttt{CPH}(A)$, $[N^\prime,g'] \gets \texttt{CPH}(B)$  \Comment{Using Algorithm~\ref{alg:find-ph}}
    \If{$|\mathrm G(N)|>|\mathrm G(N^\prime)|$}
        \State $(A, B) \gets (B, A)$, $([N,g], [N^\prime,g']) \gets ([N^\prime,g'], [N,g])$
    \EndIf
    \If{$|\mathrm G(N')|$ is small} then apply \texttt{EHEM}
    \EndIf
    \State Compute coset decomposition $\mathfrak{S}_d = \bigsqcup_{i=1}^{\ell} G(N^\prime)g_i$
    \For{$i = 1$ \textbf{to} $\ell$} \Comment{Using Theorem~\ref{thm-O-deter}}
        \If{$\mathrm{G}( \h(A \cdot g_i^{-1})) = \mathrm{G}(N^\prime)$}
            \State $[N, g''] \gets \texttt{RBLMSort}(\h(A \cdot g_i^{-1}))$  \Comment{Using Algorithm~\ref{alg:RBLM_sort}}
            \If{$N = N^\prime$}
                \State \Return $[\text{TRUE}, (A, g_i^{-1} g''), (B, g')]$    \Comment{Note: $(A \cdot g_i^{-1} g'') \simeq_{\mathrm U} (B \cdot g')$}
            \EndIf
        \EndIf
    \EndFor
    \State \Return FALSE
\end{algorithmic}
\end{algorithm}

\section{Evaluation of practical utility}

Although the aforementioned algorithms are not polynomial-time, in practical applications, they can, with a very high probability, complete the detection in polynomial time.

\subsection{Analysis of  \texttt{HEM}}

We need the probability distribution of the diagonal entries in the HNF of the matrix. Such a probability is sometimes termed natural density; for details, see e.g. \cite{NDensityDef}. Concerning the distribution of the diagonal entries of the HNF of a matrix, the following lemma is important.

\begin{prop}{\em \cite{MAZE20112398}}\label{Lemm-Probablility-HNF}
Let $h_1, h_2, \dots, h_{d-1}$ be strictly positive integers with $d \geq 2$.
The probability that a $d \times d$ nonsingular integer matrix $A$ has a  $\mathrm{HNF}$ of the form
\[
\h(A) =
\begin{pmatrix}
h_1 & * & * & * & * \\
0 & h_2 & * & * & * \\
0 & 0 & \ddots & * & * \\
0 & 0 & 0 & h_{d-1} & * \\
0 & 0 & 0 & 0 & h
\end{pmatrix},
\]
where \( h = \dfrac{\det(A)}{\prod_{i=1}^{d-1} h_i} \), is given by
\[
\bigl(\zeta(d) \cdot \zeta(d-1) \cdot \cdots \cdot \zeta(2) \cdot h_1^d \cdot h_2^{d-1} \cdot \cdots \cdot h_{d-1}^2\bigr)^{-1},
\]
where \(\zeta(s) = \sum_{n=1}^\infty n^{-s}\)  denotes the Riemann zeta function. (see, e.g.,  \cite{Zeta-Value}).
\end{prop}

Before presenting the following results, we introduce the definition of elementary symmetric functions, see, e.g., \cite[Chapter 7]{RP.Stanley2024}.
Let $X=(x_1,x_2,\ldots)$ be a set of indeterminates.
The $m$-th \emph{elementary symmetric function} $e_m(X)$ is defined using the generating function
$$E(t)=\sum_{m=0}^{\infty} e_m(X)t^m=\prod_{i=1}^{\infty}(1+x_i t).$$
We also need the following well-known formula, see, e.g., \cite[Proposition 7.8.3]{RP.Stanley2024}.
\begin{align}\label{e-si}
s_i :=e_i(1,q,q^2,\dots) = \frac{q^{i(i-1)/2}}{\prod_{j=1}^{i}(1-q^j)}.
\end{align}

\begin{lem}\label{lem-indicator-probability}
Let $m$ be an integer with $0\leq m \leq d-1$. Then the probability that the \emph{HNF} of a nonsingular matrix $A \in \mathbb{Z}^{d \times d}$ having exactly $m$ diagonal entries $\h(A)_{i,i}>1$ (where $i<d$)
is given by
\[
\mathrm{T}(d,m):= \frac{1}{\prod_{r=2}^d \zeta(r)} \cdot e_m\!\bigl(\zeta(2)-1,\zeta(3)-1,\ldots,\zeta(d)-1\bigr).
\]
where $e_m$ denotes the $m$-th elementary symmetric polynomial.

Moreover, if $N_0 = \mathrm{CPH}(A)$ then $|\mathrm{G}(N_0)| \ge (d - m - 1)!$.
\end{lem}

\begin{proof}
For the first part, the $m=0$ case follows immediately from Proposition \ref{Lemm-Probablility-HNF}.
When $m>0$, applying Proposition \ref{Lemm-Probablility-HNF} again yields
\begin{align*}
P= &\sum_{1\leq i_1<i_2<\cdots <i_m\leq d-1} \sum_{h_{i_1}, h_{i_2},\ldots,h_{i_m}\geq 2}
\left(\prod_{r=2}^d \zeta(r) \cdot h_{i_1}^{d+1-i_1} \cdot h_{i_2}^{d+1-i_2} \cdot \cdots \cdot h_{i_m}^{d+1-i_m}\right)^{-1}
\\= & \frac{1}{\prod_{r=2}^d \zeta(r)} \cdot \sum_{1\leq i_1<i_2<\cdots <i_m\leq d-1} (\zeta(d+1-i_1)-1)\cdots (\zeta(d+1-i_m)-1)
\\=& \frac{1}{\prod_{r=2}^d \zeta(r)} \cdot \sum_{2\leq i_1<i_2<\cdots < i_m\leq d} \prod_{s=1}^m (\zeta(i_s)-1)
\\=&\mathrm{T}(d,m).
\end{align*}

For the second part, set $s = d - m - 1$. We prove by induction on $d$ the claim that $N_0$ has the following form
\[
N_0 = \begin{pmatrix}
       I_{s} & * \\
       \mathbf{0} & *
    \end{pmatrix}.
\]

The base case $d = 1$ is immediate. Assume the claim holds for $d=n-1$ and we need to prove the claim for $d=n$.
The case $s=0$ is trivial, so assume $s>0$. Let $\ell$ be the leftmost $1$ in the diagonal of the HNF of $A$. Thus
\[
\h(A)= \begin{pmatrix}
          D_1 & \mathbf{0}  & * \\
          0 & 1 & * \\
          \mathbf{0} & \mathbf0 & D_2
       \end{pmatrix},
\qquad D_1 \in \mathbb{Z}^{(\ell-1)\times (\ell-1)},\;
       D_2 \in \mathbb{Z}^{(d-\ell)\times (d-\ell)}.\]

Let column $k$ be the leftmost with GCD  equal to $1$. Clearly $k\leq \ell$.
By Corollaries~\ref{cor-first-col-gcd} and~\ref{cor-block-HNF}, we obtain
\[
\h(\h(A)\cdot (1,k)) =
\begin{cases}
\begin{pmatrix}
  1 & * & 0 & * \\
  \mathbf{0} & D_1' & \mathbf0& * \\
  0 & 0 & 1 & * \\
  \mathbf 0 & \mathbf0 & \mathbf0 & D_2
\end{pmatrix}, & \text{Case 1: } k<\ell; \\[3ex]
\begin{pmatrix}
  1 & \mathbf{0} & * \\
  \mathbf0 & D_1'' & * \\
  \mathbf0 & \mathbf{0} & D_2
\end{pmatrix}, & \text{Case 2: } k=\ell.
\end{cases}
\]

In either case write $\h(\h(A)\cdot (1,k)) = \begin{pmatrix} 1 & * \\ \mathbf{0} & D \end{pmatrix}$. Applying the  hypothesis to $D$ gives
\[
\mathrm{CPH}(D)=\begin{pmatrix}
         I_{s-1} & * \\
         \mathbf{0} & *
       \end{pmatrix}.
\]
The claim then follows for $d=n$. This completes the induction and hence the proof.
\end{proof}

\begin{rem}
When iteratively computing $N_0$, each occurrence of Case~1 (as in the proof) produces an extra $1$ in the diagonal. Consequently, the cardinality $|\mathrm{G}(N_0)|$ might be much larger than $(d - m - 1)!$.
\end{rem}

Let $\mathrm P(d,k) = \sum_{m=0}^{k-1} \mathrm{T}(d,m)$ with $1\leq k\leq d$.
We call $\mathrm P(d,k)$ the \emph{indicator probability} for the set of all nonsingular $d \times d$ integer matrices with at most $k$ diagonal entries greater than $1$ in its HNF.

Note that  $e_0=1$, clearly we have
\begin{align*}
\prod_{j=2}^d \zeta(j)=\prod_{j=2}^d (1+\zeta(j)-1)=\sum_{m=0}^{d-1} e_m(\zeta(2)-1,\zeta(3)-1,\ldots,\zeta(d)-1),
\end{align*}
Therefore, we naturally get
$$\mathrm P(d,d) = \sum_{m=0}^{d-1} \mathrm{T}(d,m)=1.$$

We now turn to the lower bound of $\mathrm P(d,4)$ with $d> 20$.
\begin{lem}\label{lem-P(d,4)-pro}
For $d>20$, we have
\begin{align*}
\mathrm P(d,4)>0.999512.
\end{align*}
\end{lem}
\begin{proof}
It is clear that $\zeta(r)>1$ for $r\geq 2$.
Consequently, $\frac{1}{\prod_{r=2}^d \zeta(r)}$ is monotonically decreasing in $d$.
Since $\zeta(r)$ converges rapidly towards $1$, the function $\frac{1}{\prod_{r=2}^d \zeta(r)}$ converges rather fast to the limit $R$ with
$$R=\frac{1}{\prod_{r=2}^{\infty} \zeta(r)}=0.43575707677\ldots,$$
see \cite{MAZE20112398} or \cite[A021002]{Sloane23}.
Let
\begin{align*}
\mathrm{\widetilde{T}}(d,m)=0.43575707677 \cdot \sum_{2\leq i_1<i_2<\cdots <i_m\leq d} \prod_{s=1}^m (\zeta(i_s)-1).
\end{align*}
and
$$\mathrm{\widetilde{P}}(d,k)= \sum_{m=0}^{k-1} \mathrm{\widetilde{T}}(d,m)\quad \text{for}\quad 1\leq k\leq d.$$

Observe that $\mathrm{\widetilde{T}}(d,m_0)$ is a strictly increasing function of $d$ when $m_0$ is a fixed integer.
It follows that for fixed $k$, $\mathrm{\widetilde{P}}(d,k)$ is a strictly increasing function of $d$.
Thus, we obtain
$$\mathrm{P}(d,4)>\mathrm{\widetilde{P}}(d,4)>\mathrm{\widetilde{P}}(20,4)\quad \text{for}\quad d> 20.$$
Using \texttt{Maple} \cite{Maple}, we compute $\mathrm{\widetilde{P}}(20,4)=0.9995121935\ldots$.
This completes the proof.
\end{proof}

\begin{table}[htbp]
    	\centering
\scalebox{0.8}{\begin{tabular}{c||c|c|c|c|c}
    		\hline \hline
$d$ & $\mathrm{T}(d,0)$ & $\mathrm{T}(d,1)$ & $\mathrm{T}(d,2)$ & $\mathrm{T}(d,3)$ & $\mathrm P(d,4)$  \\
    		\hline
$57$ & 0.435757076772914 & 0.435757076772561 & 0.116499614915389 & 0.0114991965603718 &  0.9995129651  \\
    		\hline
$58$ & 0.435757076772914 &  0.435757076772561 & 0.116499614892730 & 0.0114992020364438 & 0.9995129705  \\
    		\hline
$59$ & 0.435757076772914 & 0.435757076772561 & 0.116499614889244 & 0.0114992007683036 & 0.9995129693  \\
    		\hline
$60$ & 0.435757076772914 & 0.435757076772570 & 0.116499614880529 & 0.0114992010082750 & 0.9995129695 \\
    		\hline
$61$ & 0.435757076772914 & 0.435757076772584 & 0.116499614877915 & 0.0114991907448014 &  0.9995129592 \\
            \hline
$62$ & 0.435757076772914 & 0.435757076772584 & 0.116499614881400 & 0.0114991871226579 & 0.9995129556\\
            \hline
$63$ & 0.435757076772914 & 0.435757076772584 & 0.116499614875147 & 0.0114991929368426 & 0.9995129614\\
            \hline
$64$ & 0.435757076772914 & 0.435757076772584 & 0.116499614891301 & 0.0114991983898369 & 0.9995129669\\
            \hline
$65$ & 0.435757076772914 & 0.435757076772584 & 0.116499614891271 & 0.0114992009687083& 0.9995129695\\
            \hline
\end{tabular}}
\caption{Some approximate values of $\mathrm P(d,4)$ for $57\leq d\leq 65$.}
\label{tab-d-probability}
\end{table}

\begin{rem}
     We have computed $\mathrm P(d,4$)  for $d \leq 65$, and the results for $57 \leq d \leq 65$ are presented in Table \ref{tab-d-probability}.
     It should be noted that $P(d,4)$ itself is not an increasing function of $d$.
\end{rem}

\begin{thm}\label{thm-HEM-prob}
Let $A$ and $B$ be nonsingular matrices in $\mathbb{Z}^{d \times d}$ with $|\det(A)| = |\det(B)|$.
The $\simeq_{\mathrm{UP}}$ equivalence of $A$ and $B$ can be decided by \emph{\texttt{HEM}} Algorithm~\ref{alg-HEM} in polynomial time with probability at least $1 - \epsilon$, with a time complexity of $\mathcal O (d^7 \cdot \log^2  |\det A|)$, where $\epsilon < 2.5 \times 10^{-7}$.
\end{thm}
\begin{proof}
Let $N_1=\mathrm{CPH}(A)$ and $N_2=\mathrm{CPH}(B)$. When $d>20$, by Lemmas~\ref{lem-indicator-probability} and~\ref{lem-P(d,4)-pro}, we have  the probability of $N_i$   with $\mathrm{G}(N_i)\geq (d-4)!$ is at least
$$\mathrm P(d,4)=\sum_{m=0}^{3} \mathrm{T}(d,m) > \mathrm P(20,4)=0.9995121935\ldots .$$

When $\mathrm{G}(N_1) \geq (d-4)!$ or $\mathrm{G}(N_2) \geq (d-4)!$, line~6 of \texttt{HEM} Algorithm~\ref{alg-HEM} executes at most $d!/(d-4)!$ iterations. Each iteration has time complexity \(\mathcal {O} (d^{3} \cdot \log^{2} |\det A|)\), yielding an overall complexity of \(\mathcal{O} (d^{7} \cdot \log^{2} |\det A|)\).

Since the probability that both $N_{1}$ and $N_{2}$ are simultaneously less than $(d-4)!$ is smaller than $0.0005 ^2$, the claimed bound holds.
\end{proof}

\subsection{Average-case complexity}
We establish the following average-case complexity result.
\begin{thm}\label{thm-average-poly-time}
Let $A,B$ be nonsingular matrices in $\Z^{d\times d}$ with $|\det(A)|=|\det(B)|$.
 \emph{\texttt{HEM}} Algorithm~\ref{alg-HEM} decides the UP-equivalence of $A$ and $B$.
Its worst-case time complexity is $\mathcal O(d!\cdot d^3)$, but it runs in average-case quasi-polynomial time  $\mathcal O(d^{\log d} \log^2  |\det A|)$.
\end{thm}
To prove the theorem, we need some lemmas.

\begin{lem}\label{lem-average-q}
    There exists a constant $q$ with $0 < q < 1$ such that
    \[
        \zeta(i)-1 < (\zeta(2)-1)q^{\,i-2} \quad \text{for all integers } i > 2.
    \]
\end{lem}

\begin{proof}
    For any integer $i > 2$ and every $n \ge 3$, we have $\frac{1}{n^{i}} < \frac{1}{2^{i-2} n^{2}}$.
    Therefore,
    \[
        \zeta(i)-1 < \sum_{n=2}^{\infty} \frac{1}{2^{i-2} n^{2}} = \frac{1}{2^{i-2}} (\zeta(2)-1).
    \]
    Taking $q = \frac{1}{2}$, we obtain $\zeta(i)-1 < (\zeta(2)-1) q^{\,i-2}$ for all $i > 2$.
\end{proof}

\begin{lem}\label{lem-average-poly-time}
    Let $0 < a < 1$, $0 < q < 1$, let $s_i$ be defined in Equation \eqref{e-si}.
Then we have the following asymptotic formula:
    \[
        \sum_{i=0}^{\infty} a^{i} s_{i} d^{\,i+1} = \mathcal{O}\bigl(d^{\,\log d}\bigr) \qquad \text{as } d \to \infty. 
    \]
\end{lem}
\begin{proof}
    Since $0 < q < 1$, the infinite product $Q_\infty := \prod_{j=1}^\infty (1 - q^j)$ converges to a positive constant $1/C$.
    Hence, $s_i \le C q^{i(i-1)/2}$. It follows that the summand $t_i $ satisfies
    \[
        t_i = a^i s_i d^{\,i+1} \le C d \bigl( a q^{(i-1)/2} d \bigr)^{\!i}.
    \]

Now choose $i_0 = 1 + 2\lceil \log_q(1/(2ad)) \rceil$. For large $d$, we have $i_0 = \mathcal{O}(\log d)$.
    If $i \ge i_0$, then $q^{(i-1)/2} \le 1/(2ad)$, which implies $a q^{(i-1)/2} d \le 1/2$.
    Splitting the series at $i_0$ yields
    \[
        \sum_{i=0}^\infty t_i = \sum_{i=0}^{i_0-1} t_i + \sum_{i=i_0}^\infty t_i.
    \]
    For the tail sum,
    \[
        \sum_{i=i_0}^\infty t_i \le C d \sum_{i=i_0}^\infty \left(\frac12\right)^{\!i} = \mathcal{O}(d).
    \]
    For the initial sum,
    \[
        \sum_{i=0}^{i_0-1} t_i \le C d \sum_{i=0}^{i_0-1} (ad)^{\,i} = \mathcal{O}\bigl(d (ad)^{\,i_0}\bigr).
    \]
    Since $i_0 = \mathcal{O}(\log d)$, we have $d (ad)^{\,i_0} = \mathcal{O}(d^{\log d})$, completing the proof.
\end{proof}

\begin{proof}[Proof of Theorem~\ref{thm-average-poly-time}]
  By definition, it suffices to prove that
  \begin{align*}
    \sum_{i=0}^{d-1} \mathrm{T}(d,i) \cdot (d^3 \cdot \log^2  |\det A|) \cdot \mathrm F(d,i+1) = \mathcal{O}(d^{\log d} \log^2  |\det A|),
  \end{align*}
  where $\mathrm F(d,i)$ denotes the falling factorial $d(d-1)\dotsm(d-i+1)$.

  In fact, we only need to show that
  \[
     \sum_{i=0}^{d-1} \left(e_i(\zeta(2)-1,\zeta(3)-1,\dots,\zeta(d)-1) \cdot \mathrm F(d,i+1) \right)  = \mathcal{O}(d^{\log d}).
  \]
  By Lemma~\ref{lem-average-q}, we have
  \begin{align*}
     e_i(\zeta(2)-1,\zeta(3)-1,\ldots) <  e_i(a, aq ,aq^2, \dots) = a^i e_i(1, q , q^2,\dots)=a^i s_i,
  \end{align*}
  where $a = \zeta(2)-1 < 1$ and $q = \frac{1}{2}$. Then
  \begin{align*}
    \sum_{i=0}^{\infty} a^i s_i d^{i+1}
      & > \sum_{i=0}^{\infty} e_i(\zeta(2)-1,\zeta(3)-1,\ldots) \cdot \mathrm F(d,i+1) \\
      & > \sum_{i=0}^{d-1} \left(e_i(\zeta(2)-1,\zeta(3)-1,\dots,\zeta(d)-1) \cdot \mathrm F(d,i+1) \right).
  \end{align*}

  Finally, by Lemma~\ref{lem-average-poly-time}, we conclude that
  \[
    \sum_{i=0}^{d-1} \left(e_i(\zeta(2)-1,\zeta(3)-1,\dots,\zeta(d)-1) \cdot \mathrm F(d,i+1) \right) < \sum_{i=0}^{\infty} a^i s_i d^{i+1} = \mathcal{O}(d^{\log d}),
  \]
  which completes the proof.
\end{proof}

Since testing unimodular equivalence of two full-dimensional integral simplices can be reduced to testing UP equivalence of two nonsingular matrices by Lemma~\ref{UEcone}, we now present a proof of our first contribution.

\begin{proof}[Proof of Theorem~\ref{thm-core-contri}]
    Without loss of generality, we may assume that both $\mathcal{P}$ and $\mathcal{Q}$ have the origin $\mathbf{0} \in \mathbb{Z}^d$ as a vertex, i.e., the frist column of $A$ and $B$ is $\mathbf{0}$.

    Let $M = \begin{pmatrix} \mathbf{1}^T\\ A  \end{pmatrix}  \in \mathbb{Z}^{(d+1) \times (d+1)}$, it can be written in the block form
    \[
    M = \begin{pmatrix}
        1 & \mathbf{1}^T\\
         \mathbf{0}  & A'
    \end{pmatrix},
    \]
    where $A'$ is nonsingular. By Corollary~\ref{cor-block-HNF} and Lemma~\ref{lem-indicator-probability}, the probability that  $\h(M)$ has exactly $m$ diagonal entries $\mathrm{H}(M)_{i,i}>1$ (with $i<d+1$) equals $\mathrm{T}(d,m)$, and moreover $\lvert \mathrm{CPH}(M) \rvert > (d-m)!$.

    Note that $v=|\det M|$. Then the desired conclusion now follows by arguing analogously to the proofs of Theorems~\ref{thm-HEM-prob} and~\ref{thm-average-poly-time}.
\end{proof}

\subsection{Acceleration via GCD and SNF conditions}

Since UP-equivalence implies bi-unimodular equivalence, a necessary condition for $A\simeq_{\mathrm{UP}} B$ is $\mathrm{SNF}_A=\mathrm{SNF}_B$. We employ SNF to study the UP-equivalence. Indeed, the SNF of an integer matrix is unique and computable in polynomial time even for matrices of arbitrary dimension  \cite{SNF_polynomial}. This provides an efficient computational framework.

We denote $A = (\alpha_1 \mid \alpha_2 \mid \cdots \mid \alpha_d)$ and $B = (\beta_1 \mid \beta_2 \mid \cdots \mid \beta_d)$.
For any sequence $I = (i_1, \dots, i_m)$ with distinct entries in $ \{1,\dots,d\}$ and $m \in \mathbb{Z}^+$, we denote by $A\langle I \rangle$ the submatrix of $A$ formed by the columns indexed by $I$.
When $I = (i)$ is a singleton, we simplify the notation to $A\langle i \rangle$. We also use $A\langle \hat i \rangle  = A\langle (1,\dots,i-1,i+1,\dots,d) \rangle $ for the matrix obtained from $A$ by removing the $i$-th column.

\begin{prop}\label{prop-faceUE}
    Let  $A,B\in\Z^{d\times d}$ be nonsingular matrices. If $A \simeq_{\mathrm{UP}} B$, then for any $I = (i_1, \dots, i_m)$, there exists a  $J=(j_1, \dots, j_m)$ such that
    $$
    A\langle I\rangle \simeq_{\mathrm{U}} B\langle J\rangle
    $$
and consequently $\mathrm{ SNF}(A\langle I\rangle)=\mathrm{ SNF}(B\langle J\rangle).$
\end{prop}
\begin{proof}
By assumption, there exist a $U\in\mathrm{GL}_d(\Z)$ and $\sigma \in \mathfrak S_d$ such that $UA=BP_\sigma$. This is equivalent to
$U \alpha_i = \beta_{\sigma(i)}$ for all $i$, where $\alpha_i$ and $\beta_i$ are the $i$-th column of $A$ and $BP_\sigma$, respectively.
It follows that $(UA)\langle I\rangle=(BP_\sigma) \langle I\rangle$ consists of $\beta_{j}$ for $j\in J=(\sigma(i_1),\dots,\sigma(i_m))$.
This completes the proof.
\end{proof}

Proposition~\ref{prop-faceUE} gives more necessary conditions for UP-equivalence. We only use two of them.

The first one is when $I= (i)$ is a singleton. Then $A\langle I\rangle$ is just $\alpha_i$.
Its SNF can be treated as $\mathrm{GCD}(\alpha_i)$.

Denote by $\phi(A)$ the $1\times d$ matrix of the GCDs of $A$:
$$ \phi(A) = (\mathrm{GCD}(\alpha_1), \dots, \mathrm{GCD}(\alpha_d))=(g_1,\dots, g_d).$$

\begin{cor}\label{GCD-condition}
Assume $UA=B\cdot\sigma$ with $U\in\mathrm{GL}_d(\Z)$ and $\sigma\in \mathfrak{S}_d$, then  $\phi(A)_{sorted} = \phi(B)_{sorted}$, where $[\cdot]_{\text{sorted}}$ denotes sorting the entries in  the natural order $\leq $. This equality is called the \emph{GCD} condition.
\end{cor}
\begin{proof}
Let $\phi(A)=(g_1,\dots, g_d)$ and $\phi(B)=(h_1,\dots, h_d)$. Then by Proposition~\ref{prop-faceUE},  we have
$$(\mathrm{ SNF}(A\langle 1\rangle),\ldots,\mathrm{ SNF}(A\langle d\rangle) =(\mathrm{ SNF}((B\cdot\sigma)\langle 1\rangle),\ldots,\mathrm{ SNF}((B\cdot\sigma)\langle d\rangle)).$$
It implies that $(g_1,\dots, g_d)=(h_1,\dots, h_d)\cdot \sigma$. The desired result follows.
\end{proof}
We remark that the GCD condition is trivial for simplices since each $g_i$ is $1$.

\begin{lem}
Assume that after sorting, we have $\phi(A \cdot \sigma_1) = \phi(B \cdot \sigma_2) = (h_1, \dots, h_d)$, where the sorted values satisfy
\[
h_{1} = \cdots = h_{m_1}
< h_{m_1+1} = \cdots = h_{m_1+m_2}
< \cdots <
h_{m_1+\cdots+m_{s-1}+1} = \cdots = h_{m_1+\cdots+m_{s-1}+m_s} = h_d.
\]
Define $\mathcal{S}_{\mathrm{GCD}} = \sigma_2 \mathfrak{S}_{(m_1,\dots,m_s)} \sigma_1^{-1}$. Then
\[
A \simeq_{\mathrm{UP}} B \quad \Longleftrightarrow \quad A \simeq_{\mathcal{S}_{\mathrm{GCD}}} B.
\]
The set $\mathcal{S}_{\mathrm{GCD}}$ is called the \emph{GCD set} of $A$ and $B$.
\end{lem}

\begin{proof}
The sufficiency is obvious. Now we prove the necessity.

Suppose that $UA = B \cdot \sigma$ for some $\sigma \notin \mathcal{S}_{\mathrm{GCD}}$. Write $\sigma = \sigma_2 \tau \sigma_1^{-1}$; then $\tau \notin \mathfrak{S}_{(m_1,\dots,m_s)}$. By the definition of the sorted values, we have
\[
\phi(A \cdot \sigma_1) = \phi(UA \cdot \sigma_1) = \phi((B \cdot \sigma_2) \cdot \tau).
\]
However, since $\tau$ does not preserve the block structure induced by the multiplicities $(m_1, \dots, m_s)$, it follows that $\phi((B \cdot \sigma_2) \cdot \tau) \neq (h_1, \dots, h_d)$, contradicting the assumption. Thus, the necessity is established. This completes the proof.
\end{proof}

The second one when $I = (1, \dots, i-1, i+1, \dots, d)$. Define
\[
\hat{\phi}(A) = \left( \mathrm{SNF}(A\langle \hat{1}\rangle), \mathrm{SNF}(A\langle \hat{2}\rangle), \dots, \mathrm{SNF}(A\langle \hat{d}\rangle) \right),
\]
which is a matrix with matrix entries. Results analogous to the singleton case hold here, with the total order $\leq$ replaced by the total order $\leq_{\mathrm{rlex}}$ on matrices.

\begin{cor}\label{SNF-condition}
Assume $UA=B\cdot\sigma$ with $U\in\mathrm{GL}_d(\Z)$ and $\sigma\in \mathfrak{S}_d$. Then
\[
\mathrm{SNF}(A) = \mathrm{SNF}(B)\qquad \text{ and } \qquad
\hat{\phi}(A)_{\mathrm{sorted}} = \hat{\phi}(B)_{\mathrm{sorted}}.
\]
We refer to these equalities as the \emph{SNF condition}.
\end{cor}

\begin{lem}
Assume that after sorting, we obtain $\hat{\phi}(A \cdot \sigma_1) = \hat{\phi}(B \cdot \sigma_2) = (M_1, \dots, M_d)$, where the sorted matrices satisfy
\[
M_{1} = \cdots = M_{m_1}
<_{\mathrm{rlex}} \cdots <_{\mathrm{rlex}}
M_{m_1+\cdots+m_{s-1}+1} = \cdots = M_{m_1+\cdots+m_{s-1}+m_s} = M_d.
\]
Define $\mathcal{S}_{\mathrm{SNF}} = \sigma_2 \mathfrak{S}_{(m_1,\dots,m_s)} \sigma_1^{-1}$. Then
\[
A \simeq_{\mathrm{UP}} B \quad \Longleftrightarrow \quad A \simeq_{\mathcal{S}_{\mathrm{SNF}}} B.
\]
The set $\mathcal{S}_{\mathrm{SNF}}$ is called the \emph{SNF set} of $A$ and $B$.
\end{lem}

In some cases, the above condition can be used to accelerate the \texttt{HEM} algorithm.

\begin{exa}
Let
 $$
 A=\begin{pmatrix}
1 & 1 & 1 & 1 & 1
\\0 & 2 & 2 & 2 & 2
\\0 & 0 & 4 & 4 & 4
\\0 & 0 & 0 & 8 & 8
\\0 & 0 & 0 & 0 & 32
\end{pmatrix} = \mathrm{CPH}(A),
B=\begin{pmatrix}
1 & 1 & 1 & 1 & 1
\\ 0 & 2 & 2 & 2 & 2
\\0 & 0 & 4 & 4 & 4
\\0 & 0 & 0 & 8 & 16
\\ 0 & 0 & 0 & 0 & 32
\end{pmatrix}= \mathrm{CPH}(B).
$$ Since $\mathrm{G}(A)=\mathrm{G}(B)=\{\id\}$, invoking the \texttt{HEM} algorithm is inefficient.

However, direct calculation gives
\begin{align*}
  \hat{\phi}(A)= & \left(
\begin{pmatrix}
1 & 0 & 0 & 0
\\
 0 & 4 & 0 & 0
\\
 0 & 0 & 8 & 0
\\
 0 & 0 & 0 & 32
\\
 0 & 0 & 0 & 0
\end{pmatrix}
,
\begin{pmatrix}
1 & 0 & 0 & 0
\\
 0 & 2 & 0 & 0
\\
 0 & 0 & 8 & 0
\\
 0 & 0 & 0 & 32
\\
 0 & 0 & 0 & 0
\end{pmatrix}
,
\begin{pmatrix}
1 & 0 & 0 & 0
\\
 0 & 2 & 0 & 0
\\
 0 & 0 & 4 & 0
\\
 0 & 0 & 0 & 32
\\
 0 & 0 & 0 & 0
\end{pmatrix}
,
\begin{pmatrix}
1 & 0 & 0 & 0
\\
 0 & 2 & 0 & 0
\\
 0 & 0 & 4 & 0
\\
 0 & 0 & 0 & 8
\\
 0 & 0 & 0 & 0
\end{pmatrix}
,
\begin{pmatrix}
1 & 0 & 0 & 0
\\
 0 & 2 & 0 & 0
\\
 0 & 0 & 4 & 0
\\
 0 & 0 & 0 & 8
\\
 0 & 0 & 0 & 0
\end{pmatrix}
\right), \\
   \hat{\phi}(B)= &\left(
\begin{pmatrix}
1 & 0 & 0 & 0
\\
 0 & 4 & 0 & 0
\\
 0 & 0 & 8 & 0
\\
 0 & 0 & 0 & 32
\\
 0 & 0 & 0 & 0
\end{pmatrix}
,
\begin{pmatrix}
1 & 0 & 0 & 0
\\
 0 & 2 & 0 & 0
\\
 0 & 0 & 8 & 0
\\
 0 & 0 & 0 & 32
\\
 0 & 0 & 0 & 0
\end{pmatrix}
,
\begin{pmatrix}
1 & 0 & 0 & 0
\\
 0 & 2 & 0 & 0
\\
 0 & 0 & 4 & 0
\\
 0 & 0 & 0 & 8
\\
 0 & 0 & 0 & 0
\end{pmatrix}
,
\begin{pmatrix}
1 & 0 & 0 & 0
\\
 0 & 2 & 0 & 0
\\
 0 & 0 & 4 & 0
\\
 0 & 0 & 0 & 16
\\
 0 & 0 & 0 & 0
\end{pmatrix}
,
\begin{pmatrix}
1 & 0 & 0 & 0
\\
 0 & 2 & 0 & 0
\\
 0 & 0 & 4 & 0
\\
 0 & 0 & 0 & 8
\\
 0 & 0 & 0 & 0
\end{pmatrix}
\right) .\\
\end{align*}
Since $  \hat{\phi}(A)_{\text{sorted}} \neq  \hat{\phi}(B)_{\text{sorted}}$,  $A$
and $B$ are not UP-equivalent.
\end{exa}
 
\section{Concluding remark}

We present a systematic investigation of UP equivalence, leading to the first average-case quasi-polynomial time algorithm \texttt{HEM} for deciding unimodular equivalence of $d$-dimensional integral simplices. This algorithm achieves polynomial-time complexity with a failure probability less than $2.5 \times 10^{-7}$, and represents the most efficient known approach to date.

The \emph{graph isomorphism problem}, determining whether two given graphs are isomorphic, remains a major open question in theoretical computer science, with no known polynomial-time algorithm \cite{Babai2016}. Our \texttt{HEM} algorithm offers potential for investigations into the graph isomorphism problem, a promising 
future direction.

We propose the following open problem:
\begin{prob}
Does there exist an average-case polynomial-time algorithm for detecting unimodular equivalence of $d$-dimensional integral simplices?
\end{prob}

This work also serves as an instructive example of the flexible application of Smith normal form and Hermite normal form. Further investigation into the uses of permuted Hermite normal forms appears highly worthwhile.






\noindent
{\small \textbf{Acknowledgements:}}
Guoce Xin was supported by the National Natural Science Foundation of China (Grant No.\ 12571355).


\begin{thebibliography}{99}

\bibitem{Abney-Mcpeek} F. Abney-Mcpeek, S. Biswas, S. Dutta, Y. Huang, D. Li, and N. Xu, Ehrhart-equivalence, equidecomposability, and unimodular equivalence of integral polytopes, arXiv:2101.08771v1. (2021).

\bibitem{Babai2016} L. Babai, \emph{Graph isomorphism in quasipolynomial time (extended abstract)}, STOC' 16: Proceedings of the forty-eighth annual ACM symposium on Theory of Computing (2016), 684--697.

\bibitem{BeckRobins} M. Beck and S. Robins, \emph{Computing the Continuous Discretely, in: Integer-Point Enumeration in Polyhedra}, second edition, Undergraduate Texts in Mathematics. Springer, New York, (2015).

\bibitem{HNF_survey} L. Damer, Computing the Hermite normal form: a survey, \emph{Cryptology ePrint Archive}, Paper 2024/2089, (2024).

\bibitem{Domich-Kannan-Trotter} P. D. Domich, R. Kannan, and L. E. Trotter, Hermite normal form computation using modulo determinant arithmetic, \emph{Math. Oper. Res.}, 12(1) (1987), 50--59.

\bibitem{Ehrhart62} E. Ehrhart, Sur les polyh\'edres rationnels homoth\'etiques \'a $n$ dimensions, \emph{C. R. Acad Sci. Paris.} 254 (1962), 616--618.

\bibitem{Zeta-Value} A. Erdelyi, \emph{Higher Transcendental Functions}, McGraw-Hill. Vol. 1, (1953).

\bibitem{Erbe-Haase-Santos} J. Erbe, C. Haase, and F. Santos, Ehrhart-equivalent 3-polytopes are equidecomposable, \emph{Proc. Amer. Math. Soc.} 147(12) (2019), 5373--5383.

\bibitem{Greenberg} P. Greenberg, Piecewise $\mathrm{SL}_2(\mathbb{Z})$ geometry, \emph{Trans. Amer. Math. Soc.} 335(2) (1993), 705--720.


\bibitem{Gribanov} D. V. Gribanov, \emph{Enumeration and Unimodular Equivalence of Empty Delta-Modular Simplices}, In: M. Khachay, Y. Kochetov, A. Eremeev, O. Khamisov, V. Mazalov, P. Pardalos, (eds) Mathematical Optimization Theory and Operations Research. Lecture Notes in Computer Science, vol 13930. Springer, Cham. (2023), 115--132.

\bibitem{Haase-McAllister} C. Haase and T. B. MacAllister, \emph{Quasi-period and $\mathrm{GL}_n(\mathbb{Z})$-scissors congruence in rational polytopes}, Integer points in polyhedra--geometry, number theory, representation theory, algebra, optimization, statistics, Contemp. Math., vol 452, Amer. Math. Soc., Provindence. (2008), 115--122.

\bibitem{Kaibel} V. Kaibel and A. Schwartz, On the complexity of polytope isomorphism problems, \emph{Graphs Combin.} 19 (2003), 215--230.

\bibitem{Kaibel-Pfetsch} V. Kaibel and M. Pfetsch, \emph{Some algorithm problems in polytope theory}, in the book: Algebra, Geometry, and Software Systems. Editors: M. Joswig and N. Takayama. Springer-Verlag. (2003), 23--47.

\bibitem{SNF_polynomial} R. Kannan and A. Bachem, Polynomial algorithms for computing the Smith and Hermite normal forms of an integer matrix, \emph{SIAM J. Comput.} 8(4) (1979), 499--507.

\bibitem{Luks-Graph} E. M. Luks, Isomorphism of graphs of bounded valence can be tested in polynomial time, \emph{J. Comput. Syst. Sci.} 25 (1982), 42--65.

\bibitem{Maple} Maple, Mathematics Software, (2025). Available at \url{https://cn.maplesoft.com/products/maple/index.aspx}.

\bibitem{MAZE20112398} G. Maze, Natural density distribution of Hermite normal forms of integer matrices, \emph{J. Number Theory} 131 (2011), 2398–2408.


\bibitem{Sloane23} N. J. A. Sloane et al., The On-Line Encyclopedia of Integer Sequences, (2025). Available at \url{https://oeis.org}.

\bibitem{Stanley-Vol-1} R. P. Stanley, \emph{Enumerative Combinatorics (volume 1)}, Second Edition. Cambridge Studies in Advanced Mathematics, vol. 49, Cambridge University Press, (2012).

\bibitem{RP.Stanley2024} R. P. Stanley, \emph{Enumerative Combinatorics (volume 2)}, Second Edition. Cambridge Studies in Advanced Mathematics, vol. 208, Cambridge University Press, (2024).

\bibitem{NDensityDef} G. Tenenbaum, \emph{Introduction to Analytic and Probabilistic Number Theory}, Cambridge Studies in Advanced Mathematics, vol. 46, Cambridge University Press, (1995).

\bibitem{Turner-Wu} P. Turner and Y. Wu, Discrete equidecomposability and Ehrhart theory of polygons, \emph{Discrete Comput. Geom.} 65 (2021), 90--115.


\end{thebibliography}
\end{document}